\documentclass[12pt, a4paper, abstracton]{scrartcl}
\pdfoutput=1

\usepackage{amsmath, amsthm, amsfonts, amssymb}
\usepackage[utf8]{inputenc}
\usepackage[T1]{fontenc}
\usepackage[nottoc]{tocbibind} 
\usepackage{color}
\usepackage{slashed} 
\usepackage[all, cmtip]{xy}
\usepackage{hyperref} 
\usepackage{chngcntr} 
\usepackage{graphicx} 
\usepackage[english]{babel}
\usepackage{microtype}
\usepackage{bm} 
\usepackage{mathtools} 

\newcommand{\IU}{\mathcal{U}^\ast_{-\infty}}

\newcommand{\IN}{\mathbb{N}}
\newcommand{\IZ}{\mathbb{Z}}

\newcommand{\IR}{\mathbb{R}}
\newcommand{\IC}{\mathbb{C}}

\newcommand{\injrad}{\operatorname{inj-rad}}
\newcommand{\Rm}{\operatorname{Rm}}

\newcommand{\LLip}{L\text{-}\operatorname{Lip}}
\newcommand{\IB}{\mathfrak{B}}
\newcommand{\IK}{\mathfrak{K}}

\newcommand{\supp}{\operatorname{supp}}
\newcommand{\diam}{\operatorname{diam}}
\newcommand{\Hom}{\operatorname{Hom}}

\newcommand{\id}{\operatorname{id}}
\newcommand{\kernel}{\operatorname{ker}}

\newcommand{\ind}{\operatorname{ind}}

\newcommand{\op}{\mathrm{op}}

\newcommand{\Symb}{\operatorname{Symb}}

\newcommand{\UPsiDO}{\mathrm{U}\Psi\mathrm{DO}}

\newcommand{\Frechet}{Fr\'{e}chet }

\newcommand{\Spakula}{\v{S}pakula }

\theoremstyle{plain}
\newtheorem{thm}{Theorem}[section]
\newtheorem*{thm*}{Theorem}
\newtheorem*{mainthm*}{Main Result}
\newtheorem{cor}[thm]{Corollary}
\newtheorem*{cor*}{Corollary}
\newtheorem{lem}[thm]{Lemma}
\newtheorem{prop}[thm]{Proposition}

\newtheorem{question}[thm]{Question}

\theoremstyle{definition}
\newtheorem{defn-alt}[thm]{Definition}
\newtheorem{example-alt}[thm]{Example}
\newtheorem{examples-alt}[thm]{Examples}
\newtheorem{rem-alt}[thm]{Remark}
\newtheorem{nota-alt}[thm]{Notation}

\newenvironment{defn}    
{
	\pushQED{\qed}\begin{defn-alt}}
	{\popQED\end{defn-alt}}

\newenvironment{example}    
{
	\pushQED{\qed}\begin{example-alt}}
	{\popQED\end{example-alt}}
	
\newenvironment{examples}    
{
	\pushQED{\qed}\begin{examples-alt}}
	{\popQED\end{examples-alt}}

\newenvironment{rem}    
{
	\pushQED{\qed}\begin{rem-alt}}
	{\popQED\end{rem-alt}}

\numberwithin{equation}{section}

\counterwithout{footnote}{section}

\makeatletter
\def\blfootnote{\gdef\@thefnmark{}\@footnotetext}
\makeatother

\begin{document}

\title{$K$-homology classes of elliptic uniform pseudodifferential operators}
\author{Alexander Engel}
\date{}
\maketitle

\vspace*{-3.5\baselineskip}
\begin{center}
\footnotesize{
\textit{
Fakult\"{a}t f\"{u}r Mathematik\\
Universit\"{a}t Regensburg\\
93040 Regensburg, GERMANY\\
\href{mailto:alexander.engel@mathematik.uni-regensburg.de}{alexander.engel@mathematik.uni-regensburg.de}
}}
\end{center}

\vspace*{-0.5\baselineskip}
\begin{abstract}
We show that an elliptic uniform pseudodifferential operator over a manifold of bounded geometry defines a class in uniform $K$-homology, and that this class only depends on the principal symbol of the operator. 
\end{abstract}

\tableofcontents

\section{Introduction}

Pseudodifferential operators are an indispensable tool in the study of elliptic differential operators (like Dirac operators) and their index theory. The calculus of pseudodifferential operators on compact manifolds encompasses parametrices of elliptic differential operators, i.e., their inverses up to smoothing operators, which enables one to deduce the usual important results about elliptic operators like elliptic regularity. Also, the first published proof of the Atiyah--Singer index theorem \cite{atiyah_singer_1} goes via pseudodifferential operators.

The first goal of the present paper is to set up a suitable calculus of pseudodifferential operators on non-compact manifolds. It turns out that the only for us useful definition of such a calculus is the uniform one, and that such a definition is only possible on manifolds of bounded geometry. Let us explain why on non-compact manifolds we have to consider uniform pseudodifferential operators. Recall that on $\IR^m$ an operator $P$ is called pseudodifferential, if it is given by
\[(Pu)(x) = (2 \pi)^{-n/2} \int_{\IR^m} e^{i \langle x, \xi\rangle} p(x, \xi) \hat u(\xi) \ d\xi,\]
where $\hat u$ denotes the Fourier transform of $u$ and the function $p(x,\xi)$ satisfies for some $k \in\IZ$ the estimates $\|D_x^\alpha D_\xi^\beta p(x,\xi)\| \le C^{\alpha \beta} (1 + |\xi|)^{k - |\beta|}$ for all multi-indices $\alpha$ and $\beta$. On manifolds one calls an operator pseudodifferential if one has the above representation in any local chart. But if the manifold is not compact, we get the problem that this is not sufficient to guarantee that the operator has continuous extensions to Sobolev spaces.\footnote{We are ignoring in this discussion the fact that on non-compact manifolds we also need a condition on the behaviour of the integral kernel of $P$ at infinity.} For this we additionally have to require that the above bounds $C^{\alpha \beta}$ are uniform across all the local charts. But since this is not well-defined (choosing a different atlas may distort the bounds arbitrarily large across the charts of the atlas), we will have to restrict the charts to exponential charts and additionally we will have to assume that our manifold has bounded geometry (these restrictions become clear when one looks at Lemma \ref{lem:transition_functions_uniformly_bounded}).

Our calculus of pseudodifferential operators on manifolds of bounded geometry is not totally new: Kordyukov \cite{kordyukov}, Shubin \cite{shubin} and Taylor \cite{taylor_pseudodifferential_operators_lectures} already developed similar calculi. Let us explain the difference between their calculi and the one we develop in this paper. The underlying reason that different versions of such calculi are possible is due to the fact that on non-compact manifolds one needs to control the behaviour of the integral kernels of these operators at infinity. One possibility is to impose finite propagation, i.e., demanding that there is an $R > 0$ such that the integral kernel $k(x,y)$ of the pseudodifferential operator vanishes for all $x,y$ with $d(x,y) > R$ (recall that pseudodifferential operators always have an integral kernel that is smooth outside the diagonal). More generally, one can require an exponential decay of the integral kernel at infinity, and usually this decay should be faster than the volume growth of the manifold. In the present article we will require that our pseudodifferential operators are quasilocal\footnote{An operator $A \colon H^r(E) \to H^s(F)$ is \emph{quasilocal}, if there exists some function $\mu\colon \IR_{> 0} \to \IR_{\ge 0}$ with $\mu(R) \to 0$ for $R \to \infty$ and such that for all $L \subset M$ and all $u \in H^r(E)$ with $\supp u \subset L$ we have $\|A u\|_{H^s, M - B_R(L)} \le \mu(R) \cdot \|u\|_{H^r}$.}, since this seems to be in a certain sense the most general notion which we may impose (see, e.g., the proof of Corollary \ref{cor:quasilocal_neg_order_uniformly_locally_compact} for how quasi-locality is used).

Let us explain why we want our operators to be quasilocal. To construct the large scale index class of an operator $D$ of Dirac type, we have to consider the operator $f(D)$, where $f$ is a Schwartz function with $f(0) = 1$. Now usually $f(D)$ will not have finite propagation, but it will be a quasilocal operator. This was proven by Roe for operators of Dirac type \cite[Section~5]{roe_index_1} and we will generalize this crucial fact to elliptic uniform pseudodifferential operators, see Corollary~\ref{cor:schwartz_function_of_PDO_quasilocal_smoothing}. So even though we could restrict to finite propagation uniform pseudodifferential operators and use the fact that $f(P)$ will be quasilocal whenever we need, we would leave our class of finite propagation operators in this way. So working from the beginning with quasilocal operators leads to the fact that we never have to leave this class. Note that the proof of the fact that $f(P)$ is quasilocal requires substantial analysis and is one of our key technical lemmas. It relies on a close analysis of the propagation properties of the wave operators $e^{itP}$ associated to an elliptic uniform pseudodifferential operator $P$, see Lemma~\ref{lem:exp(itP)_quasilocal}.

The approach to index theory preferred by the author is the one via $K$-homology. This is a generalized homology theory in which elliptic operators naturally define classes. On non-compact manifolds of bounded geometry an important branch of index theory (large scale index theory) is investigated via, e.g., the rough assembly map $K_\ast^u(-) \to K_\ast(C^\ast_u -)$ developed by \Spakula (the coarse setting is actually more common,\footnote{Coarse index theory was mainly developed by Roe \cite{roe_coarse_cohomology,roe_index_coarse,roe_lectures_coarse_geometry}.} but since we have to work with uniform operators anyway the more natural approach is to use the rough theory which is as a uniform version of the coarse one). Here $K_\ast^u(-)$ denotes uniform $K$-homology and $K_\ast(C^\ast_u -)$ is the $K$-theory of the uniform Roe algebra. Given an operator of Dirac type over a manifold of bounded geometry, \Spakula showed \cite[Section~3]{spakula_uniform_k_homology} that it has a class in uniform $K$-homology and therefore large scale index theory can be applied. Our investigations of uniform pseudodifferential operators on manifolds of bounded geometry and our analysis of the propagation properties of functions of such operators enables us to generalize the result of \Spakula and show that elliptic uniform pseudodifferential operators also have classes in uniform $K$-homology. For the resulting index theory see \cite{engel_rough} and \cite[Section 5]{engel_indices_UPDO}.

\begin{mainthm*}[Theorem~\ref{thm:elliptic_symmetric_PDO_defines_uniform_Fredholm_module} and Proposition~\ref{prop:same_symbol_same_k_hom_class}]
Let $P$ be a symmetric and elliptic uniform pseudodifferential operator acting on a vector bundle of bounded geometry over a manifold $M$ of bounded geometry.

Then $P$ defines naturally a uniform $K$-homology class $[P] \in K_\ast^u(M)$ and this class does only depend on the principal symbol of $P$.
\end{mainthm*}

\paragraph{Acknowledgements} This article contains mostly Section~2 of the preprint \cite{engel_indices_UPDO} which is being split up for easier publication. It arose out of the Ph.D.\ thesis \cite{engel_phd} of the author written at the University of Augsburg.

\section{Bounded geometry}

We will define in this section the notion of bounded geometry for manifolds and for vector bundles and discuss basic facts about uniform $C^r$-spaces and Sobolev spaces on them.

\begin{defn}
We will say that a Riemannian manifold $M$ has \emph{bounded geometry}, if
\begin{itemize}
\item the curvature tensor and all its derivatives are bounded, i.e., $\| \nabla^k \Rm (x) \| < C_k$ for all $x \in M$ and $k \in \IN_0$, and
\item the injectivity radius is uniformly positive, i.e., $\injrad_M(x) > \varepsilon > 0$ for all points $x \in M$ and for a fixed $\varepsilon > 0$.
\end{itemize}
If $E \to M$ is a vector bundle with a metric and compatible connection, we say that \emph{$E$ has bounded geometry}, if the curvature tensor of $E$ and all its derivatives are bounded.
\end{defn}

\begin{examples}
There are plenty of examples of manifolds of bounded geometry. The most important ones are coverings of compact Riemannian manifolds equipped with the pull-back metric, homogeneous manifolds with an invariant metric, and leafs in a foliation of a compact Riemannian manifold (this is proved by Greene in \cite[lemma on page 91 and the paragraph thereafter]{greene}).

For vector bundles, the most important examples are of course again pull-back bundles of bundles over compact manifolds equipped with the pull-back metric and connection, and the tangent bundle of a manifold of bounded geometry.

Furthermore, if $E$ and $F$ are two vector bundles of bounded geometry, then the dual bundle $E^\ast$, the direct sum $E \oplus F$, the tensor product $E \otimes F$ (and so especially also the homomorphism bundle $\Hom(E, F) = F \otimes E^\ast$) and all exterior powers $\Lambda^l E$ are also of bounded geometry. If $E$ is defined over $M$ and $F$ over $N$, then their external tensor product\footnote{The fiber of $E \boxtimes F$ over the point $(x,y) \in M \times N$ is given by $E_x \otimes F_y$.} $E \boxtimes F$ over $M \times N$ is also of bounded geometry.
\end{examples}

Greene proved in \cite[Theorem 2']{greene} that there are no obstructions against admitting a metric of bounded geometry, i.e., every smooth manifold without boundary admits one. On manifolds of bounded geometry there is also no obstruction for a vector bundle to admit a metric and compatible connection of bounded geometry. The proof (i.e., the construction of the metric and the connection) is done in a uniform covering of $M$ by normal coordinate charts and subordinate uniform partition of unity (we will discuss these things in a moment) and we have to use the local characterization of bounded geometry for vector bundles from Lemma \ref{lem:equiv_characterizations_bounded_geom_bundles}.

We will now state an important characterization in local coordinates of bounded geometry since it will allow us to show that certain local definitions (like the one of uniform pseudodifferential operators) are independent of the chosen normal coordinates.

\begin{lem}[{\cite[Appendix A1.1]{shubin}}]\label{lem:transition_functions_uniformly_bounded}
Let the injectivity radius of $M$ be positive.

Then the curvature tensor of $M$ and all its derivatives are bounded if and only if for any $0 < r < \injrad_M$ all the transition functions between overlapping normal coordinate charts of radius $r$ are uniformly bounded, as are all their derivatives (i.e., the bounds can be chosen to be the same for all transition functions).
\end{lem}

Another fact which we will need about manifolds of bounded geometry is the existence of uniform covers by normal coordinate charts and corresponding partitions of unity. A proof may be found in, e.g., \cite[Appendix A1.1]{shubin} (Shubin addresses the first statement about the existence of such covers actually to the paper \cite{gromov_curvature_diameter_betti_numbers} of Gromov).

\begin{lem}\label{lem:nice_coverings_partitions_of_unity}
Let $M$ be a manifold of bounded geometry.

For every $0 < \varepsilon < \tfrac{\injrad_M}{3}$ there exists a covering of $M$ by normal coordinate charts of radius $\varepsilon$ with the properties that the midpoints of the charts form a uniformly discrete set in $M$ and that the coordinate charts with double radius $2\varepsilon$ form a uniformly locally finite cover of $M$.

Furthermore, there is a subordinate partition of unity $1 = \sum_i \varphi_i$ with $\supp \varphi_i \subset B_{2\varepsilon}(x_i)$, such that in normal coordinates the functions $\varphi_i$ and all their derivatives are uniformly bounded (i.e., the bounds do not depend on $i$).
\end{lem}

If the manifold $M$ has bounded geometry, we have analogous equivalent local characterizations of bounded geometry for vector bundles as for manifolds. The equivalence of the first two bullet points in the next lemma is stated in, e.g., \cite[Proposition 2.5]{roe_index_1}. Concerning the third bullet point, the author could not find any citable reference in the literature (though Shubin uses in \cite{shubin} this as the actual definition).

\begin{lem}\label{lem:equiv_characterizations_bounded_geom_bundles}
Let $M$ be a manifold of bounded geometry and $E \to M$ a vector bundle. Then the following are equivalent:

\begin{itemize}
\item $E$ has bounded geometry,
\item the Christoffel symbols $\Gamma_{i \alpha}^\beta(y)$ of $E$ with respect to synchronous framings (considered as functions on the domain $B$ of normal coordinates at all points) are bounded, as are all their derivatives, and this bounds are independent of $x \in M$, $y \in \exp_x(B)$ and $i, \alpha, \beta$, and
\item the matrix transition functions between overlapping synchronous framings are uniformly bounded, as are all their derivatives (i.e., the bounds are the same for all transition functions).
\end{itemize}
\end{lem}

We will now give the definition of uniform $C^\infty$-spaces together with a local characterization on manifolds of bounded geometry. The interested reader is refered to, e.g., the papers \cite[Section 2]{roe_index_1} or \cite[Appendix A1.1]{shubin} of Roe and Shubin for more information regarding these uniform $C^\infty$-spaces.

\begin{defn}[$C^r$-bounded functions]
Let $f \in C^\infty(M)$. We say that $f$ is a \emph{$C_b^r$-function}, or equivalently that it is \emph{$C^r$-bounded}, if $\| \nabla^i f \|_\infty < C_i$ for all $0 \le i \le r$.
\end{defn}

If $M$ has bounded geometry, being $C^r$-bounded is equivalent to the statement that in every normal coordinate chart $|\partial^\alpha f(y)| < C_\alpha$ for every multiindex $\alpha$ with $|\alpha| \le r$ (where the constants $C_\alpha$ are independent of the chart).

Of course, the definition of $C^r$-boundedness and its equivalent characterization in normal coordinate charts for manifolds of bounded geometry make also sense for sections of vector bundles of bounded geometry (and so especially also for vector fields, differential forms and other tensor fields).

\begin{defn}[Uniform $C^\infty$-spaces]
\label{defn:uniform_frechet_spaces}
Let $E$ be a vector bundle of bounded geometry over $M$. We will denote the \emph{uniform $C^r$-space} of all $C^r$-bounded sections of $E$ by $C_b^r(E)$.

Furthermore, we define the \emph{uniform $C^\infty$-space $C_b^\infty(E)$}
\[C_b^\infty(E) := \bigcap_r C_b^r(E)\]
which is a \Frechet space.
\end{defn}

Now we get to Sobolev spaces on manifolds of bounded geometry. Much of the following material is from \cite[Appendix A1.1]{shubin} and \cite[Section 2]{roe_index_1}, where an interested reader can find more thorough discussions of this matters.

Let $s \in C^\infty_c(E)$ be a compactly supported, smooth section of some vector bundle $E \to M$ with metric and connection $\nabla$. For $k \in \IN_0$ and $p \in [1,\infty)$ we define the global $W^{k,p}$-Sobolev norm of $s$ by
\begin{equation}\label{eq:sobolev_norm}
\|s\|_{W^{k,p}}^p := \sum_{i=0}^k \int_M \|\nabla^i s(x)\|^p dx.
\end{equation}

\begin{defn}[Sobolev spaces $W^{k,p}(E)$]\label{defn:sobolev_spaces}
Let $E$ be a vector bundle which is equipped with a metric and a connection. The \emph{$W^{k,p}$-Sobolev space of $E$} is the completion of $C^\infty_c(E)$ in the norm $\|-\|_{W^{k,p}}$ and will be denoted by $W^{k,p}(E)$.
\end{defn}

If $E$ and $M^m$ both have bounded geometry than the Sobolev norm \eqref{eq:sobolev_norm} for $1 < p < \infty$ is equivalent to the local one given by
\begin{equation}\label{eq:sobolev_norm_local}
\|s\|_{W^{k,p}}^p \stackrel{\text{equiv}}= \sum_{i=1}^\infty \|\varphi_i s\|^p_{W^{k,p}(B_{2\varepsilon}(x_i))},
\end{equation}
where the balls $B_{2\varepsilon}(x_i)$ and the subordinate partition of unity $\varphi_i$ are as in Lemma \ref{lem:nice_coverings_partitions_of_unity}, we have chosen synchronous framings and $\|-\|_{W^{k,p}(B_{2\varepsilon}(x_i))}$ denotes the usual Sobolev norm on $B_{2\varepsilon}(x_i) \subset \IR^m$. This equivalence enables us to define the Sobolev norms for all $k \in \IR$, see Triebel \cite{triebel_2} and Gro{\ss}e--Schneider \cite{grosse_sobolev}. There are some issues in the case $p=1$, see the discussion by Triebel \cite[Section~2.2.3]{triebel_1}, \cite[Remark~4 on Page~13]{triebel_2}.

Assuming bounded geometry, the usual embedding theorems are true:

\begin{thm}[{\cite[Theorem 2.21]{aubin_nonlinear_problems}}]\label{thm:sobolev_embedding}
Let $E$ be a vector bundle of bounded geometry over a manifold $M^m$ of bounded geometry and without boundary.

Then we have for all values $(k-r)/m > 1/p$ continuous embeddings
\[W^{k,p}(E) \subset C^r_b(E).\]
\end{thm}

We define the space
\begin{equation}
\label{eq:defn_W_infty}
W^{\infty,p}(E) := \bigcap_{k \in \IN_0} W^{k,p}(E)
\end{equation}
and equip it with the obvious \Frechet topology. The Sobolev Embedding Theorem tells us now that we have for all $p$ a continuous embedding
\[W^{\infty,p}(E) \hookrightarrow C^\infty_b(E).\]

For $p=2$ we will write $H^k(E)$ for $W^{k,2}(E)$. This are Hilbert spaces and for $k < 0$ the space $H^k(E)$ coincides with the dual of $H^{-k}(E)$, regarded as a space of distributional sections of $E$.

We will now investigate the Sobolev spaces $H^\infty(E)$ and $H^{-\infty}(E)$ of infinite orders. They are crucial since they will allow us to define smoothing operators and hence the important algebra $\IU(E)$ in the next section.

\begin{lem}
The topological dual of $H^\infty(E)$ is given by
\[H^{-\infty}(E) := \bigcup_{k \in \IN_0} H^{-k}(E).\]
\end{lem}

Let us equip the space $H^{-\infty}(E)$ with the locally convex topology defined as follows: the \Frechet space $H^\infty(E) = \operatorname{\underleftarrow{\lim}} H^k(E)$ is the projective limit of the Banach spaces $H^k(E)$, so using dualization we may put on the space $H^{-\infty}(E)$ the inductive limit topology denoted $\iota(H^{-\infty}(E), H^\infty(E))$:
\[H^{-\infty}_\iota(E) := \operatorname{\underrightarrow{\lim}} H^{-k}(E).\]
It enjoys the following universal property: a linear map $A \colon H^{-\infty}_\iota(E) \to F$ to a locally convex topological vector space $F$ is continuous if and only if $A|_{H^{-k}(E)}\colon H^{-k}(E) \to F$ is continuous for all $k \in \IN_0$.

Later we will need to know how the bounded\footnote{A subset $B \subset H^{-\infty}_\iota(E)$ is \emph{bounded} if and only if for all open neighbourhoods $U \subset H^{-\infty}_\iota(E)$ of $0$ there exists $\lambda > 0$ with $B \subset \lambda U$.} subsets of $H^{-\infty}_\iota(E)$ look like, which is the content of the following lemma. In its proof we will also deduce some nice properties of the spaces $H^\infty(E)$ and $H^{-\infty}_\iota(E)$.

\begin{lem}\label{lem:regular_inductive_limit}
The space $H^{-\infty}_\iota(E) := \operatorname{\underrightarrow{\lim}} H^{-k}(E)$ is a \emph{regular inductive limit}, i.e., for every bounded subset $B \subset H^{-\infty}_\iota(E)$ exists some $k \in \IN_0$ such that $B$ is already contained in $H^{-k}(E)$ and bounded there.\footnote{Note that the converse does always hold for inductive limits, i.e., if $B \subset H^{-k}(E)$ is bounded, then it is also bounded in $H^{-\infty}_\iota(E)$.}
\end{lem}

\begin{proof}
Since all $H^{-k}(E)$ are \Frechet spaces, we may apply the following corollary of Grothendieck's Factorization Theorem: the inductive limit $H^{-\infty}_\iota(E)$ is regular if and only if it is locally complete (see, e.g., \cite[Lemma 7.3.3(i)]{perezcarreras_bonet}). To avoid introducing more burdensome vocabulary, we won't define the notion of local completeness here since we will show something stronger: $H^{-\infty}_\iota(E)$ is actually complete\footnote{That is to say, every Cauchy net converges. In locally convex spaces, being Cauchy and to converge is meant with respect to each of the seminorms simultaneously.}.

From \cite[Sections 3.(a \& b)]{bierstedt_bonet} we conclude the following: since each $H^k(E)$ is a Hilbert space, the \Frechet space $H^\infty(E)$ is the projective limit of reflexive Banach spaces and therefore totally reflexive\footnote{That is to say, every quotient of it is reflexive, i.e., the canonical embeddings of the quotients into their strong biduals are isomorphisms of topological vector spaces.}. It follows that $H^\infty(E)$ is distinguished, which can be characterized by $H^{-\infty}_\beta(E) = H^{-\infty}_\iota(E)$, where $\beta(H^{-\infty}(E), H^\infty(E))$ is the strong topology on $H^{-\infty}(E)$. Now without defining the strong topology we just note that strong dual spaces of \Frechet space are always complete.
\end{proof}

\section{Quasilocal smoothing operators}\label{sec:quasiloc}

We will discuss in this section the definition and basic properties of smoothing operators on manifolds of bounded geometry and we will introduce the notion of quasilocal operators. The quasilocal smoothing operators will be the $(-\infty)$-part of our uniform pseudodifferential operators that we are going to define in the next section.

\begin{defn}[Smoothing operators]
Let $M$ be a manifold of bounded geometry and $E$ and $F$ two vector bundles of bounded geometry over $M$. We will call a continuous linear operator $A \colon H^{-\infty}_\iota(E) \to H^\infty(F)$ a \emph{smoothing operator}.
\end{defn}

\begin{lem}\label{lem:smoothing_operator_iff_bounded}
A linear operator $A \colon H^{-\infty}_\iota(E) \to H^\infty(F)$ is continuous if and only if it is bounded as an operator $H^{-k}(E) \to H^l(F)$ for all $k, l \in \IN_0$.
\end{lem}

Let us denote by $\IB(H^{-\infty}_\iota(E), H^\infty(E))$ the algebra of all smoothing operators from $E$ to itself. Due to the above lemma we may equip it with the countable family of norms $(\|-\|_{-k,l})_{k,l \in \IN_0}$ so that it becomes a \Frechet space\footnote{That is to say, a topological vector space whose topology is Hausdorff and induced by a countable family of seminorms such that it is complete with respect to this family of seminorms.}.

Now let us get to the main property of smoothing operators that we will need, namely that they can be represented as integral operators with a uniformly bounded smooth kernel. Let $A \colon H^{-\infty}_\iota(E) \to H^\infty(F)$ be given. Then we get by the Sobolev Embedding Theorem \ref{thm:sobolev_embedding} a continuous operator $A \colon H^{-\infty}_\iota(E) \to C_b^\infty(F)$ and so may conclude by the Schwartz Kernel Theorem for regularizing operators\footnote{Note that the usual wording of the Schwartz Kernel Theorem for regularizing operators requires the domain $H^{-\infty}(E)$ to be equipped with the weak-$^\ast$ topology $\sigma(H^{-\infty}(E), H^{\infty}(F))$ and $A$ to be continuous against it. But one actually only needs the domain to be equipped with the inductive limit topology. To see this, one can look at the proof of the Schwartz Kernel Theorem for regularizing kernels as in, e.g., \cite[Theorem 3.18]{ganglberger}.} that $A$ has a smooth integral kernel $k_A \in C^\infty(F \boxtimes E^\ast)$, which is uniformly bounded as are all its derivatives, because of the bounded geometry of $M$ and the vector bundles $E$ and $F$, i.e., $k_A \in C_b^\infty(F \boxtimes E^\ast)$.

From the proof of the Schwartz Kernel Theorem for regularizing operators we also see that the assignment of the kernel to the operator is continuous against the \Frechet topology on $\IB(H^{-\infty}_\iota(E), H^\infty(F))$. Furthermore, because of Lemma \ref{lem:regular_inductive_limit} this topology coincides with the topology of bounded convergence\footnote{A basis of neighbourhoods of zero for the topology of bounded convergence is given by the subsets $M(B, U) \subset \IB(H^{-\infty}_\iota(E), H^\infty(F))$ of all operators $T$ with $T(B) \subset U$, where $B$ ranges over all bounded subsets of $H^{-\infty}_\iota(E)$ and $U$ over a basis of neighbourhoods of zero in $H^\infty(F)$.} on $\IB(H^{-\infty}_\iota(E), H^\infty(F))$. We need this equality of topologies in order for the next proposition (which is a standard result in distribution theory) to be equivalent to the version stated in \cite[Proposition 2.9]{roe_index_1}.

\begin{prop}\label{prop:smoothing_op_kernel}
Let $A\colon H^{-\infty}_\iota(E) \to H^\infty(F)$ be a smoothing operator. Then $A$ can be written as an integral operator with kernel $k_A \in C_b^\infty(F \boxtimes E^\ast)$. Furthermore, the map
\[\IB(H^{-\infty}_\iota(E), H^\infty(F)) \to C_b^\infty(F \boxtimes E^\ast),\]
which associates a smoothing operator its kernel, is continuous.
\end{prop}

Let $L \subset M$ be any subset. We will denote by $\|-\|_{H^r, L}$ the seminorm on the Sobolev space $H^r(E)$ given by
\[\|u\|_{H^r, L} := \inf \{ \|u^\prime\|_{H^r} \ | \ u^\prime \in H^r(E), u^\prime = u \text{ on a neighbourhood of }L\}.\]

\begin{defn}[Quasilocal operators, {\cite[Section 5]{roe_index_1}}]\label{defn:quasiloc_ops}
We will call a continuous operator $A \colon H^r(E) \to H^s(F)$ \emph{quasilocal}, if there is a function $\mu\colon \IR_{> 0} \to \IR_{\ge 0}$ with $\mu(R) \to 0$ for $R \to \infty$ and such that for all $L \subset M$ and all $u \in H^r(E)$ with $\supp u \subset L$ we have
\[\|A u\|_{H^s, M - B_R(L)} \le \mu(R) \cdot \|u\|_{H^r}.\]

Such a function $\mu$ will be called a \emph{dominating function for $A$}.

We will say that an operator $A\colon C_c^\infty(E) \to C^\infty(F)$ is a \emph{quasilocal operator of order $k$}\footnote{Roe calls such operators ``\emph{uniform} operators of order $k$'' in \cite[Definition 5.3]{roe_index_1}. But since the word ``uniform'' will have another meaning for us (see, e.g., the definition of uniform $K$-homology), we changed the name.} for some $k \in \IZ$, if $A$ has a continuous extension to a quasilocal operator $H^s(E) \to H^{s-k}(F)$ for all $s \in \IZ$.

A smoothing operator $A\colon H^{-\infty}_\iota(E) \to H^\infty(F)$ will be called \emph{quasilocal}, if $A$ is quasilocal as an operator $H^{-k}(E) \to H^l(F)$ for all $k,l \in \IN_0$ (from which it follows that $A$ is also quasilocal for all $k,l \in \IZ$).
\end{defn}

If we regard a smoothing operator $A$ as an operator $L^2(E) \to L^2(F)$, we get a uniquely defined adjoint $A^\ast\colon L^2(F) \to L^2(E)$. Its integral kernel will be given by
\[k_{A^\ast}(x,y) := k_A(y,x)^\ast \in C_b^\infty(E \boxtimes F^\ast),\]
where $k_A(y,x)^\ast \in F_y^\ast \otimes E_x$ is the dual element of $k_A(y,x) \in F_y \otimes E^\ast_x$.

\begin{defn}[cf.~{\cite[Definition 5.3]{roe_index_1}}]\label{defn:quasiloc_smoothing}
We will denote the set of all quasilocal smoothing operators $A\colon H^{-\infty}_\iota(E) \to H^\infty(F)$ with the property that their adjoint operator $A^\ast$ is also a quasilocal smoothing operator $H^{-\infty}_\iota(F) \to H^\infty(E)$ by $\IU(E,F)$.

If $E=F$, we will just write $\IU(E)$.
\end{defn}

\begin{rem}
Roe defines in \cite[Definition 5.3]{roe_index_1} the algebra $\mathcal{U}_{-\infty}(E)$ instead of $\IU(E)$, i.e., he does not demand that the adjoint operator is also quasilocal smoothing. The reason why we do this is that we want adjoints of uniform pseudodifferential operators to be again uniform pseudodifferential operators (and the algebra $\IU(E)$ is used in the definition of uniform pseudodifferential operators).
\end{rem}

\section{Uniform pseudodifferential operators}

Let $M^m$ be an $m$-dimensional manifold of bounded geometry and let $E$ and $F$ be two vector bundles of bounded geometry over $M$. Now we will get to the definition of uniform pseudodifferential operators acting on sections of vector bundles of bounded geometry over manifolds of bounded geometry.

Our definition is almost the same as the ones of Shubin \cite{shubin} and Kordyukov \cite{kordyukov}. The difference is that our definition is slightly more general, because we do not restrict to finite propagation operators (since we allow the term $P_{-\infty}$ in the definition below). The reason why we have to do this is because of our results in Section~\ref{sec:functions_of_PDOs}: we only get quasi-local operators and not necessarily operators of finite propagation (in fact, the main technical result is Lemma~\ref{lem:exp(itP)_quasilocal} stating that the wave operators are quasi-local), and therefore we would leave our calculus of pseudodifferential operators if we would insist of them having finite propagation. Most of the results stated in this section and in Section~\ref{secio23ed} are basically already known, resp., it is straight-forward to generalize the corresponding statements in the finite propagation case to our quasi-local case. We nevertheless include a discussion of these statements in order for our exposition here to be self-contained.

\begin{defn}\label{defn:pseudodiff_operator}
An operator $P\colon C_c^\infty(E) \to C^\infty(F)$ is a \emph{uniform pseudodifferential operator of order $k \in \IZ$}, if with respect to a uniformly locally finite covering $\{B_{2\varepsilon}(x_i)\}$ of $M$ with normal coordinate balls and corresponding subordinate partition of unity $\{\varphi_i\}$ as in Lemma \ref{lem:nice_coverings_partitions_of_unity} we can write
\begin{equation}
\label{eq:defn_pseudodiff_operator_sum}
P = P_{-\infty} + \sum_i P_i
\end{equation}
satisfying the following conditions:
\begin{itemize}
\item $P_{-\infty} \in \IU(E,F)$, i.e., it is a quasilocal smoothing operator,
\item for all $i$ the operator $P_i$ is with respect to synchronous framings of $E$ and $F$ in the ball $B_{2\varepsilon}(x_i)$ a matrix of pseudodifferential operators on $\IR^m$ of order $k$ with support\footnote{An operator $P$ is \emph{supported in a subset $K$}, if $\supp Pu \subset K$ for all $u$ in the domain of $P$ and if $Pu = 0$ whenever we have $\supp u \cap K = \emptyset$.} in $B_{2\varepsilon}(0) \subset \IR^m$, and
\item the constants $C_i^{\alpha \beta}$ appearing in the bounds
\[\|D_x^\alpha D_\xi^\beta p_i(x,\xi)\| \le C^{\alpha \beta}_i (1 + |\xi|)^{k - |\beta|}\]
of the symbols of the operators $P_i$ can be chosen to not depend on $i$, i.e., there are $C^{\alpha \beta} < \infty$ such that
\begin{equation}
\label{eq:uniformity_defn_PDOs}
C^{\alpha \beta}_i \le C^{\alpha \beta}
\end{equation}
for all multi-indices $\alpha, \beta$ and all $i$. We will call this the \emph{uniformity condition} for pseudodifferential operators on manifolds of bounded geometry.
\end{itemize}
We denote the set of all such operators by $\UPsiDO^k(E,F)$.
\end{defn}

\begin{rem}
From Lemma \ref{lem:transition_functions_uniformly_bounded} and Lemma \ref{lem:equiv_characterizations_bounded_geom_bundles} together with \cite[Theorem III.§3.12]{lawson_michelsohn} (and its proof which gives the concrete formula how the symbol of a pseudodifferential operator transforms under a coordinate change) we conclude that the above definition of uniform pseudodifferential operators on manifolds of bounded geometry does neither depend on the chosen uniformly locally finite covering of $M$ by normal coordinate balls, nor on the subordinate partition of unity with uniformly bounded derivatives, nor on the synchronous framings of $E$ and $F$.
\end{rem}

\begin{rem}
We could also have given an equivalent definition of uniform pseudodifferential operators, which does not need a choice of covering: firstly, for each $\varepsilon > 0$ there should be a quasilocal smoothing operator $P_\varepsilon$ such that for any $\phi, \psi \in C_c^\infty(M)$ with $d(\supp \phi, \supp \psi) > \varepsilon$ and any $v \in C_c^\infty(E)$ we have $\psi \cdot P(\phi \cdot v) = \psi \cdot P_\varepsilon(\phi \cdot v)$. This encodes that the integral kernel of a uniform pseudodifferential operator $P$ is off-diagonally a quasilocal smoothing operator.

Secondly, to encode the behaviour of the integral kernel of $P$ at its diagonal, we must demand that in any normal coordinate chart of radius less than the injectivity radius of the manifold with any choice of cut-off function for this coordinate chart and with any choice of synchronous framings of $E$ and $F$ in this coordinate chart the operator $P$ looks like a pseudodifferential operator on $\IR^m$, and for the collection of all of these local representatives of $P$ computed with respect to cut-off functions that have common bounds on their derivatives we have the Uniformity Condition \eqref{eq:uniformity_defn_PDOs}.
\end{rem}

\begin{rem}
We consider only operators that would correspond to H\"ormander's class $S_{1, 0}^k(\Omega)$, if we consider open subsets $\Omega$ of $\IR^m$ instead of an $m$-dimensional manifold $M$, i.e., we do not investigate operators corresponding to the more general classes $S_{\rho, \delta}^k(\Omega)$. The paper \cite[Definition 2.1]{hormander_ess_norm} is the one where H\"ormander introduced these classes.
\end{rem}

Recall that in the case of compact manifolds a pseudodifferential operator $P$ of order~$k$ has an extension to a continuous operator $H^s(E) \to H^{s-k}(F)$ for all $s \in \IZ$ (see, e.g., \cite[Theorem III.§3.17(i)]{lawson_michelsohn}). Due to the uniform local finiteness of the sum in \eqref{eq:defn_pseudodiff_operator_sum} and due to the Uniformity Condition \eqref{eq:uniformity_defn_PDOs}, this result does also hold in our case of a manifold of bounded geometry.

\begin{prop}\label{prop:pseudodiff_extension_sobolev}
Let $P \in \UPsiDO^k(E,F)$. Then $P$ has for all $s \in \IZ$ an extension to a continuous operator $P\colon H^s(E) \to H^{s-k}(F)$.
\end{prop}

\begin{rem}\label{rem:bound_operator_norm_PDO}
Later we will need the following fact: we can bound the operator norm of $P\colon H^s(E) \to H^{s-k}(F)$ from above by the maximum of the constants $C^{\alpha 0}$ with $|\alpha| \le K_s$ from the Uniformity Condition \eqref{eq:uniformity_defn_PDOs} for $P$ multiplied with a constant $C_s$, where $K_s \in \IN_0$ and $C_s$ only depend on $s \in \IZ$ and the dimension of the manifold $M$. This can be seen by carefully examining the proof of \cite[Proposition III.§3.2]{lawson_michelsohn} which is the above proposition for the compact case.\footnote{To be utterly concrete, we have to choose normal coordinate charts and a subordinate partition of unity as in Lemma \ref{lem:nice_coverings_partitions_of_unity} and also synchronous framings for $E$ and $F$ and then use Formula \eqref{eq:sobolev_norm_local} which gives Sobolev norms that can be computed locally and that are equivalent to the global norms \eqref{eq:sobolev_norm}.}
\end{rem}

Let us define
\[\UPsiDO^{-\infty}(E,F) := \bigcap_k \UPsiDO^k(E,F).\]

We will show $\UPsiDO^{-\infty}(E,F) = \IU(E,F)$: from the previous Proposition \ref{prop:pseudodiff_extension_sobolev} we conclude that $P \in \UPsiDO^{-\infty}(E,F)$ is a smoothing operator (using Lemma \ref{lem:smoothing_operator_iff_bounded}). Since we can write $P = P_{-\infty} + \sum_i P_i$, where $P_{-\infty} \in \IU(E,F)$ and the $P_i$ are supported in balls with uniformly bounded radii, the operator $\sum_i P_i$ is of finite propagation. So $P$ is the sum of a quasilocal smoothing operator $P_{-\infty}$ and a smoothing operator $\sum_i P_i$ of finite propagation, and therefore a quasilocal smoothing operator. The same arguments also apply to the adjoint $P^\ast$ of $P$, so that in the end we can conclude $P \in \IU(E,F)$, i.e., we have shown $\UPsiDO^{-\infty}(E,F) \subset \IU(E,F)$.

Since the other inclusion does hold by definition, we get the claim.\footnote{Of course, our definition of pseudodifferential operators was arranged such that this lemma holds.}

\begin{lem}\label{lem:PDO_-infinity_equal_quasilocal_smoothing}
$\UPsiDO^{-\infty}(E,F) = \IU(E,F)$.
\end{lem}

One of the important properties of pseudodifferential operators on compact manifolds is that the composition of an operator $P \in \UPsiDO^k(E,F)$ and $Q \in \UPsiDO^l(F,G)$ is again a pseudodifferential operator of order $k+l$: $PQ \in \UPsiDO^{k+l}(E,G)$. We can prove this also in our setting by writing
\begin{align*}
PQ & = \Big(P_{-\infty} + \sum_i P_i\Big) \Big(Q_{-\infty} + \sum_j Q_j\Big)\\
& = P_{-\infty} Q_{-\infty} + \sum_i P_i Q_{-\infty} + \sum_j P_{-\infty} Q_j + \sum_{i,j} P_i Q_j
\end{align*}
and then arguing as follows.
\begin{itemize}
\item The first summand is an element of $\IU(E,G)$: in \cite[Proposition 5.2]{roe_index_1} it was shown that the composition of two quasilocal operators is again quasilocal and it is clear that composing smoothing operators again gives smoothing operators, resp. it is easy to see that composing two operators which may be approximated by finite propagation operators again gives such an operator.

\item The second and third summands are from $\IU(E,G)$ due to Proposition \ref{prop:pseudodiff_extension_sobolev} and since the sums are uniformly locally finite, the operators $P_i$ and $Q_j$ are supported in coordinate balls of uniform radii (i.e., have finite propagation which is uniformly bounded from above) and their operator norms are uniformly bounded due to the uniformity condition in the definition of pseudodifferential operators.

\item The last summand is a uniformly locally finite sum of pseudodifferential operators of order $k+l$ (here we use the corresponding result for compact manifolds) and to see the Uniformity Condition \eqref{eq:uniformity_defn_PDOs} we use \cite[Theorem III.§3.10]{lawson_michelsohn}: it states that the symbol of $P_i Q_j$ has formal development $\sum_\alpha \frac{i^{|\alpha|}}{\alpha !} (D_\xi^\alpha p_i)(D_x^\alpha q_j)$. So we may deduce the uniformity condition for $P_i Q_j$ from the one for $P_i$ and for $Q_j$.
\end{itemize}

Other properties that immediately generalize from the compact to the bounded geometry case is firstly, that the commutator of two uniform pseudodifferential operators whose symbols commute (Definition~\ref{defnnkdf893}) is of one order lower than it should a priori be, and secondly, that multiplication with a function $f \in C_b^\infty(M)$ defines a uniform pseudodifferential operator of order $0$.

So we have the following important proposition:

\begin{prop}\label{prop:PsiDOs_filtered_algebra}
$\UPsiDO^\ast(E)$ is a filtered $^\ast$-algebra, i.e., for all $k, l \in\IZ $ we have
\[\UPsiDO^k(E) \circ \UPsiDO^l(E) \subset \UPsiDO^{k+l}(E),\]
and so $\UPsiDO^{-\infty}(E)$ is a two-sided $^\ast$-ideal in $\UPsiDO^\ast(E)$.

Furthermore, we have $[P, Q] \in \UPsiDO^{k+l-1}(E)$ for $P \in \UPsiDO^k(E)$, $Q \in \UPsiDO^l(E)$, $k,l \in \IZ$, provided the symbols of $P$ and $Q$ commute.

Moreover, multiplication with a function $f \in C_b^\infty(M)$ defines a uniform pseudodifferential operator of order $0$ whose symbol commutes with any other symbol.
\end{prop}

The last property that generalizes to our setting and that we want to mention is the following (the proof of \cite[Theorem III.§3.9]{lawson_michelsohn} generalizes directly):

\begin{prop}\label{prop:Pu_smooth_if_u_smooth}
Let $P \in \UPsiDO^k(E,F)$ be a uniform pseudodifferential operator of arbitrary order and let $u \in H^s(E)$ for some $s \in \IZ$.

Then, if $u$ is smooth on some open subset $U \subset M$, $Pu$ is also smooth on $U$.
\end{prop}

\section{Uniformity of operators of non-positive order}\label{sec:uniformity_PDOs}

Now we get to the important statement that the uniform pseudodifferential operators we have defined are, in fact, ``uniform'' in the meaning to be defined now (the discussion here is strongly related to the fact that symmetric and elliptic uniform pseudodifferential operators will define uniform $K$-homology classes).

Let $T \in \IK(L^2(E))$ be a compact operator. We know that $T$ is the limit of finite rank operators, i.e., for every $\varepsilon > 0$ there is a finite rank operator $k$ such that $\|T - k\| < \varepsilon$. Now given a collection $\mathcal{A} \subset \IK(L^2(E))$ of compact operators, it may happen that for every $\varepsilon > 0$ the rank needed to approximate an operator from $\mathcal{A}$ may be bounded from above by a common bound for all operators. This is formalized in the following definition.

\begin{defn}[Uniformly approximable collections of operators]\label{defn:uniformly_approximable_collection}
A collection of operators $\mathcal{A} \subset \IK(L^2(E))$ is said to be \emph{uniformly approximable}, if for every $\varepsilon > 0$ there is an $N > 0$ such that for every $T \in \mathcal{A}$ there is a rank-$N$ operator $k$ with $\|T - k\| < \varepsilon$.
\end{defn}

\begin{examples}\label{ex:uniformly_approximable_collections}
Every collection of finite rank operators with uniformly bounded rank is uniformly approximable.

Furthermore, every finite collection of compact operators is uniformly approximable and so also every totally bounded subset of $\IK(L^2(E))$.

The converse is in general false since a uniformly approximable family need not be bounded (take infinitely many rank-$1$ operators with operator norms going to infinity).

Even if we assume that the uniformly approximable family is bounded we do not necessarily get a totally bounded set: let $(e_i)_{i \in \IN}$ be an orthonormal basis of $L^2(E)$ and $P_i$ the orthogonal projection onto the $1$-dimensional subspace spanned by the vector $e_i$. Then the collection $\{P_i\} \subset \IK(L^2(E))$ is uniformly approximable (since all operators are of rank $1$) but not totally bounded (since $\|P_i - P_j\| = 1$ for $i \not= j$)\footnote{Another way to see that $\{P_i\}$ is not totally bounded is to use the characterization of totally bounded subsets of $\IK(H)$ from \cite[Theorem 3.5]{anselone_palmer}: a family $\mathcal{A} \subset \IK(H)$ is totally bounded if and only if both $\mathcal{A}$ and $\mathcal{A}^\ast$ are collectively compact, i.e., the sets $\{T v \ | \ T \in \mathcal{A}, v \in H \text{ with } \|v\| = 1\} \subset H$ and $\{T^\ast v \ | \ T \in \mathcal{A}, v \in H \text{ with } \|v\| = 1\} \subset H$ have compact closure.}.
\end{examples}

Let us define
\begin{equation*}\label{defn:L-Lip_R(M)}
\LLip_R(M) := \{ f \in C_c(M) \ | \ f \text{ is }L\text{-Lipschitz}, \diam(\supp f) \le R \text{ and } \|f\|_\infty \le 1\}.
\end{equation*}

\begin{defn}[{\cite[Definition 2.3]{spakula_uniform_k_homology}}]\label{defn:uniform_operators_manifold}
Let $T \in \IB(L^2(E))$. We say that $T$ is \emph{uniformly locally compact}, if for every $R, L > 0$ the collection
\[\{fT, Tf \ | \ f \in \LLip_R(M)\}\]
is uniformly approximable.

We say that $T$ is \emph{uniformly pseudolocal}, if for every $R, L > 0$ the collection
\[\{[T, f] \ | \ f \in \LLip_R(M)\}\]
is uniformly approximable.
\end{defn}

We will now show that uniform pseudodifferential operators of negative order are uniformly locally compact and that uniform pseudodifferential operators of order $0$ are uniformly pseudolocal. We will start with the operators of negative order.

\begin{prop}\label{prop:quasilocal_negative_order_uniformly_locally_compact}
Let $A\in \IB(L^2(E))$ be a finite propagation operator of negative order $k < 0$\footnote{See Definition \ref{defn:quasiloc_ops}. Note that we do not assume that $A$ is a pseudodifferential operator.} such that its adjoint operator $A^\ast$ also has finite propagation and is of negative order $k^\prime < 0$. Then $A$ is uniformly locally compact. Even more, the collection
\[\{fT, Tf \ | \ f \in B_R(M)\}\]
is uniformly approximable for every $R > 0$, where $B_R(M)$ consists of all bounded Borel functions $h$ on $M$ with $\diam(\supp h) \le R$ and $\|h\|_\infty \le 1$.
\end{prop}

\begin{proof}
Let $f \in B_R(M)$, $K := \supp f \subset M$ and $r$ be the propagation of $A$. The operator $\chi A f = A f$, where $\chi$ is the characteristic function of $B_r(K)$, factores as
\[L^2(E) \stackrel{\cdot f}\longrightarrow L^2(E|_K) \stackrel{\chi \cdot A}\longrightarrow H^{-k}(E|_{B_r(K)}) \hookrightarrow L^2(E|_{B_r(K)}) \to L^2(E).\]
The following properties hold:
\begin{itemize}
\item multiplication with $f$ has operator norm $\le 1$, since $\|f\|_\infty \le 1$, and analogously for the multiplication with $\chi$,
\item the norm of $\chi \cdot A\colon L^2(E|_K) \to H^{-k}(E|_{B_r(K)})$ can be bounded from above by the norm of $A\colon L^2(E) \to H^{-k}(E)$ (i.e., the upper bound does not depend on $K$ nor $r$),
\item the inclusion $H^{-k}(E|_{B_r(K)}) \hookrightarrow L^2(E|_{B_r(K)})$ is compact (due to  the Theorem of Rellich--Kondrachov) and this compactness is uniform, i.e., its approximability by finite rank operators\footnote{Here we mean the existence of an upper bound on the rank needed to approximate the operator by finite rank operators, given an $\varepsilon > 0$.} depends only on $R$ (the upper bound for the diameter of $\supp f$) and $r$, but not on $K$ (this uniformity is due to the bounded geometry of $M$ and of the bundles $E$ and $F$), and
\item the inclusion $L^2(E|_{B_r(K)}) \to L^2(E)$ is of norm $\le 1$.
\end{itemize}
From this we conclude that the operator $\chi A f = A f$ is compact and this compactness is uniform, i.e., its approximability by finite rank operators depends only on $R$ and $r$. So we can conclude that $\{Af \ | \ f \in B_R(M)\}$ is uniformly approximable.

Applying the same reasoning to the adjoint operator,\footnote{By assumption the adjoint operator also has finite propagation and is of negative order. So we conclude that $\{A^\ast f \ | \ f \in B_R(M)\}$ is uniformly approximable. But a collection $\mathcal{A}$ of compact operators is uniformly approximable if and only if the adjoint family $\mathcal{A}^\ast$ is uniformly approximable. So we get that $\{(A^\ast f)^\ast = \overline{f} A \ | \ f \in B_R(M)\}$ is uniformly approximable.} we conclude that $A$ is uniformly locally compact.
\end{proof}

Using an approximation argument\footnote{Note that we will not approximate the quasilocal operator $A$ itself by finite propagation operators in this argument. In fact, it is an open problem whether quasilocal operators may be approximated by finite propagation operators; see Section \ref{sec_final_remarks}.} we may also show the following corollary:

\begin{cor}\label{cor:quasilocal_neg_order_uniformly_locally_compact}
Let $A$ be a quasilocal operator of negative order and let the same hold true for its adjoint $A^\ast$. Then $A$ is uniformly locally compact; in fact, it even satisfies the stronger condition from the above Proposition \ref{prop:quasilocal_negative_order_uniformly_locally_compact}.
\end{cor}

\begin{proof}
We have to show that $\{Af \ | \ f \in B_R(M)\}$ is uniformly approximable. Let $\varepsilon > 0$ be given and let $r_\varepsilon$ be such that $\mu_A(r) < \varepsilon$ for all $r \ge r_\varepsilon$, where $\mu_A$ is the dominating function of $A$. Then $\chi_{B_{r_\varepsilon}(\supp f)} A f$ is $\varepsilon$-away from $Af$ and the same reasoning as in the proof of the above Proposition \ref{prop:quasilocal_negative_order_uniformly_locally_compact} shows that the approximability (up to an error of $\varepsilon$) of $\chi_{B_{r_\varepsilon}(\supp f)} A f$ does only depend on $R$ and $r_\varepsilon$. From this the claim that $\{Af \ | \ f \in B_R(M)\}$ is uniformly approximable follows.

Using the adjoint operator and the same arguments for it, we conclude that $A$ is uniformly locally compact.
\end{proof}

\begin{cor}\label{prop:PsiDOs_negative_order_uniformly_locally_compact}
Let $P \in \UPsiDO^k(E)$ be a uniform pseudodifferential operator of negative order $k < 0$. Then $P$ is uniformly locally compact.
\end{cor}

Let us now get to the case of uniform pseudodifferential operators of order $0$, where we want to show that such operators are uniformly pseudolocal.

Recall the following fact for compact manifolds: $T$ is pseudolocal\footnote{That is to say, $[T, f]$ is a compact operator for all $f \in C_c(M)$.} if and only if $f T g$ is a compact operator for all $f, g \in C(M)$ with disjoint supports. This observation is due to Kasparov and a proof might be found in, e.g., \cite[Proposition 5.4.7]{higson_roe}. We can add another equivalent characterization which is basically also proved in the cited proposition: an operator $T$ is pseudolocal if and only if $f T g$ is a compact operator for all bounded Borel functions $f$ and $g$ on $M$ with disjoint supports.

We have analogous equivalent characterizations for uniformly pseudolocal operators, which we will state in the following lemma. The proof of it is similar to the compact case (and uses the fact that the subset of all uniformly pseudolocal operators is closed in operator norm, which is proved in \cite[Lemma 4.2]{spakula_uniform_k_homology}). Furthermore, in order to prove that the Points 4 and 5 in the statement of the next lemma are equivalent to the other points we need the bounded geometry of $M$. For the convenience of the reader we will give a full proof of the lemma.

Let us introduce the notions $B_b(M)$ for all bounded Borel functions on $M$ and $B_R(M)$ for its subset consisting of all function $h$ with $\diam(\supp h) \le R$ and $\|h\|_\infty \le 1$.

\begin{lem}\label{lem:kasparov_lemma_uniform_approx_manifold}
The following are equivalent for an operator $T \in \IB(L^2(E))$:
\begin{enumerate}
\item $T$ is uniformly pseudolocal,
\item for all $R, L > 0$ the following collection is uniformly approximable:
\[\{f T g, g T f \ | \ f \in B_b(M), \ \! \|f\|_\infty \le 1, \ \! g \in \LLip_R(M), \ \! \supp f \cap \supp g = \emptyset\},\]
\item for all $R, L > 0$ the following collection is uniformly approximable:
\[\{f T g, g T f \ | \ f \in B_b(M), \ \! \|f\|_\infty \le 1, \ \! g \in B_R(M), \ \! d(\supp f, \supp g) \ge L\},\]
\item for every $L > 0$ there is a sequence $(L_j)_{j \in \IN}$ of positive numbers (not depending on the operator $T$) such that
\begin{align*}
\{ f T g, g T f \ | \ & f \in B_b(M)\text{ with }\|f\|_\infty \le 1,\\
& g \in B_R(M) \cap C_b^\infty(M)\text{ with }\|\nabla^j g\|_\infty \le L_j,\text{ and}\\
& \supp f \cap \supp g = \emptyset\}
\end{align*}
is uniformly approximable for all $R, L > 0$.
\item for every $L > 0$ there is a sequence $(L_j)_{j \in \IN}$ of positive numbers (not depending on the operator $T$) such that
\[\{ [T,g] \ | \ g \in B_R(M) \cap C_b^\infty(M)\text{ with }\|\nabla^j g\|_\infty \le L_j\}\]
is uniformly approximable for all $R, L > 0$.
\end{enumerate}
\end{lem}

\begin{proof}
\bm{$1 \Rightarrow 2$}\textbf{:} Let $f \in B_b(M)$ with $\|f \|_\infty \le 1$ and $g \in \LLip_R(M)$ have disjoint supports, i.e., $\supp f \cap \supp g = \emptyset$. From the latter we conclude $f T g = f [T,g]$, from which the claim follows (because $T$ is uniformly pseudolocal and because the operator norm of multiplication with $f$ is $\le 1$). Of course such an argument also works with the roles of $f$ and $g$ changed.

\bm{$2 \Rightarrow 3$}\textbf{:} Let $f \in B_b(M)$ with $\|f \|_\infty \le 1$ and $g \in B_R(M)$ with $d(\supp f, \supp g) \ge L$. We define $g^\prime(x) := \max\big( 0, 1 - 1/L \cdot d(x, \supp g) \big) \in 1/L\text{-}\operatorname{Lip}_{R+2L}(M)$. Since $g^\prime g = g$, the claim follows from writing $f T g = f T g^\prime g$ and because multiplication with $g$ has operator norm $\le 1$, and we of course also may change the roles of $f$ and $g$.

\bm{$3 \Rightarrow 1$}\textbf{:} Let $f \in \LLip_R(M)$. For given $\varepsilon > 0$ we partition the range of $f$ into a sequence of non-overlapping half-open intervals $U_1, \ldots, U_n$, each having diameter less than $\varepsilon$, such that $\overline{U_i}$ intersects $\overline{U_j}$ if and only if $|i - j| \le 1$. Denoting by $\chi_i$ the characteristic function of $f^{-1}(U_i)$, we get that $\chi_i \in B_R(M)$ if $0 \notin U_i$, since the support of $f$ has diameter less than or equal to $R$, and furthermore $d(\supp \chi_i, \supp \chi_j) \ge \tfrac{\varepsilon}{L}$ if $|i-j| > 1$, since $f$ is $L$-Lipschitz.

By Point 3 we have that the collections $\{\chi_i T \chi_j, \chi_j T \chi_i\}$ are uniformly approximable for all $i,j$ with $|i-j| > 1$. Choosing points $x_1, \ldots, x_n$ from $f^{-1}(U_1), \ldots, f^{-1}(U_n)$ and defining $f^\prime := f(x_1) \chi_1 + \cdots + f(x_n) \chi_n$, we get $\|f - f^\prime\|_\infty < \varepsilon$. The operator $[T,f]$ is $2\varepsilon \|T\|$-away from $[T, f^\prime]$, and since $\chi_1 + \cdots + \chi_n = 1$ we have
\[Tf^\prime - f^\prime T = \sum_{i,j} \chi_j T f(x_i)\chi_i - f(x_j) \chi_j T \chi_i.\]
Since we already know that $\{\chi_i T \chi_j, \chi_j T \chi_i\}$ are uniformly approximable for all $i,j$ with $|i-j| > 1$, it remains to treat the sum (note that the summand for $i=j$ is zero)
\[\sum_{|i-j|=1} \chi_j T f(x_i)\chi_i - f(x_j) \chi_j T \chi_i = \sum_{|i-j|=1} \big( f(x_i) - f(x_j) \big) \chi_j T \chi_i.\]
We split the sum into two parts, one where $i = j+1$ and the other one where $i=j-1$. The first part takes the form
\[\sum_j \big( f(x_{j+1}) - f(x_j) \big) \chi_j T \chi_{j+1},\]
i.e., is a direct sum of operators from $\chi_{j+1} \cdot L^2(E)$ to $\chi_j \cdot L^2(E)$. Therefore its norm is the maximum of the norms of its summands. But the latter are $\le 2 \varepsilon \|T\|$ since $|f(x_{j+1}) - f(x_j)| \le 2\varepsilon$. We treat the second part of the sum in the above display the same way and conclude that the sum in the above display is in norm $\le 4\varepsilon T$. Putting it all together it follows that $T$ is the operator norm limit of uniformly pseudolocal operators, from which it follows that $T$ itself is uniformly pseudolocal (it is proved in \cite[Lemma 4.2]{spakula_uniform_k_homology} that the uniformly pseudolocal operators are closed in operator norm, as are also the uniformly locally compact ones).

\bm{$2 \Rightarrow 4$}\textbf{:} Clear. We have to set $L_1 := L$ and the other values $L_{j \ge 2}$ do not matter (i.e., may be set to something arbitrary).

\bm{$4 \Rightarrow 3$}\textbf{:} This is similar to the proof of $2 \Rightarrow 3$, but we have to smooth the function $g^\prime$ constructed there. Let us make this concrete, i.e., let $f \in B_b(M)$ with $\|f \|_\infty \le 1$ and $g \in B_R(M)$ with $d(\supp f, \supp g) \ge L$ be given. We define
\[g^\prime(x) := \max\big( 0, 1 - 2 / L \cdot d(x, B_{L / 4}(\supp g)) \big) \in 2/L\text{-}\operatorname{Lip}_{R+3L/2}(M).\]
Note that $g^\prime \equiv 1$ on $B_{L/4}(\supp g)$ and $g^\prime \equiv 0$ outside $B_{3L/4}(\supp g)$. We cover $M$ by normal coordinate charts and choose a ``nice'' subordinate partition of unity $\varphi_i$ as in Lemma \ref{lem:nice_coverings_partitions_of_unity}. If $\psi$ is now a mollifier on $\IR^m$ supported in $B_{L/8}(0)$, we apply it in every normal coordinate chart to $\varphi_i g^\prime$ and reassemble then all the mollified parts of $g^\prime$ again to a (now smooth) function $g^\prime{}^\prime$ on $M$. This function $g^\prime{}^\prime$ is now supported in $B_{7L/8}(\supp g)$, and is constantly $1$ on $B_{L/8}(\supp g)$. So $f T g = f T g^\prime{}^\prime g$ from which we may conclude the uniform approximability of the collection $\{f T g\}$ for $f$ and $g$ satisfying $f \in B_b(M)$ with $\|f \|_\infty \le 1$ and $g \in B_R(M)$ with $d(\supp f, \supp g) \ge L$. Note that the constants $L_j$ appearing in $\|\nabla^j g^\prime{}^\prime\|_\infty \le L_j$ depend on $L$, $\varphi_i$ and $\psi$, but not on $f$, $g$ or $R$. The dependence on $\varphi_i$ and $\psi$ is ok, since we may just fix a particular choice of them (note that the choice of $\psi$ also depends on $L$), and the dependence on $L$ is explicitly stated in the claim.

Of course we may also change the roles of $f$ and $g$ in this argument.

\bm{$5 \Rightarrow 4$}\textbf{:} Clear. We just have to write $fTg = f[T,g]$ and analogously for $gTf$.

\bm{$1 \Rightarrow 5$}\textbf{:} Clear.
\end{proof}

With the above lemma at our disposal we may now prove the following proposition.

\begin{prop}\label{prop:PDO_order_0_l-uniformly-pseudolocal}
Let $P \in \UPsiDO^0(E)$. Then $P$ is uniformly pseudolocal.
\end{prop}

\begin{proof}
Writing $P = P_{-\infty} + \sum_i P_i$ with $P_{-\infty} \in \IU(E)$, we may without loss of generality assume that $P$ has finite propagation $R^\prime$ (since $P_{-\infty}$ is uniformly locally compact by the above Corollary \ref{cor:quasilocal_neg_order_uniformly_locally_compact} and uniformly locally compact operators are uniformly pseudolocal).

We will use the equivalent characterization in Point 4 of the above lemma: let $R, L > 0$ and the corresponding sequence $(L_j)_{j \in \IN}$ be given. We have to show that
\begin{align*}
\{ f P g, g P f \ | \ & f \in B_b(M)\text{ with }\|f\|_\infty \le 1,\\
& g \in B_R(M) \cap C_b^\infty(M)\text{ with }\|\nabla^j g\|_\infty \le L_j,\text{ and}\\
& \supp f \cap \supp g = \emptyset\}
\end{align*}
is uniformly approximable for all $R, L > 0$.

We have
\[f P g = f \chi_{B_{R^\prime}(\supp g)} P g = f \chi_{B_{R^\prime}(\supp g)} [P, g]\]
since the supports of $f$ and $g$ are disjoint.

With Proposition \ref{prop:PsiDOs_filtered_algebra} we conclude that multiplication with $g$ is a uniform pseudodifferential operator of order $0$ (since $g \in C_b^\infty(M)$) and furthermore, that the commutator $[P, g]$ is a pseudodifferential operator of order $-1$. Therefore, by Corollary \ref{prop:PsiDOs_negative_order_uniformly_locally_compact}, we know that the set $\{f \chi_{B_{R^\prime}(\supp g)} [P, g] \ | \ f \in B_R(M)\}$ is uniformly approximable. So we conclude that our operators $f[P,g]$ have the needed uniformity in the functions $f$.

It remains to show that we also have the needed uniformity in the functions $g$. Writing $P = \sum_i P_i$\footnote{Recall that we assumed without loss of generality that there is no $P_{-\infty}$.}, we get $[P,g] = \sum_i [P_i,g]$. Now each $[P_i, g]$ is a uniform pseudodifferential operator of order $-1$, their supports\footnote{Recall that an operator $P$ is \emph{supported in a subset $K$}, if $\supp Pu \subset K$ for all $u$ in the domain of $P$ and if $Pu = 0$ whenever we have $\supp u \cap K = \emptyset$.} depend only on the propagation of $P$ and on the value of $R$ (but not on $i$ nor on the concrete choice of $g$) and their operator norms as maps $L^2(E) \to H^1(E)$ are bounded from above by a constant that only depends on $P$, on $R$ and on the values of all the $L_j$ (but again, neither on $i$ nor on $g$). The last fact follows from a combination of Remark \ref{rem:bound_operator_norm_PDO} together with the estimates on the symbols of the $[P_i,g]$ that we get from the proof that they are uniform pseudodifferential operators of order $-1$. So examining the proof of Proposition \ref{prop:quasilocal_negative_order_uniformly_locally_compact} more closely, we see that these properties suffice to conclude the needed uniformity of $f[P,g]$ in the functions $g$.

The operators $g P f$ may be treated analogously.
\end{proof}

\section{Elliptic operators}\label{secio23ed}
In this section we will define the notion of ellipticity\footnote{It is actually \emph{uniform ellipticity} that we define here. But since non-uniform ellipticity is not a natural notion for uniform pseudodifferential operators, we just call it \emph{ellipticity} what we define.} for uniform pseudodifferential operators and discuss important consequences of it (elliptic regularity, fundamental elliptic estimates and essential self-adjointness). Most of the results are already known and can be found in the literature (at least in the case of finite propagation operators). We nevertheless include a discussion of them so that our exposition here is self-contained.

Let $\pi^\ast E$ and $\pi^\ast F$ denote the pull-back bundles of $E$ and $F$ to the cotangent bundle $\pi\colon T^\ast M \to M$ of the $m$-dimensional manifold $M$.

\begin{defn}[Symbols]\label{defnnkdf893}
Let $p$ be a section of the bundle $\Hom(\pi^\ast E, \pi^\ast F)$ over $T^\ast M$. We call $p$ a \emph{symbol of order $k \in \IZ$}, if the following holds: choosing a uniformly locally finite covering $\{ B_{2 \varepsilon}(x_i) \}$ of $M$ through normal coordinate balls and corresponding subordinate partition of unity $\{ \varphi_i \}$ as in Lemma \ref{lem:nice_coverings_partitions_of_unity}, and choosing synchronous framings of $E$ and $F$ in these balls $B_{2\varepsilon}(x_i)$, we can write $p$ as a uniformly locally finite sum $p = \sum_i p_i$, where $p_i(x,\xi) := p(x,\xi) \varphi(x)$ for $x \in M$ and $\xi \in T^\ast_x M$, and interpret each $p_i$ as a matrix-valued function on $B_{2 \varepsilon}(x_i) \times \IC^m$. Then for all multi-indices $\alpha$ and $\beta$ there must exist a constant $C^{\alpha \beta} < \infty$ such that for all $i$ and all $x, \xi$ we have
\begin{equation}\label{eq:symbol_uniformity}
\|D^\alpha_x D^\beta_\xi p_i(x,\xi) \| \le C^{\alpha \beta}(1 + |\xi|)^{k - |\beta|}.
\end{equation}
We denote the vector space all symbols of order $k \in \IZ$ by $\Symb^k(E,F)$.
\end{defn}

From Lemma \ref{lem:transition_functions_uniformly_bounded} and Lemma \ref{lem:equiv_characterizations_bounded_geom_bundles} we conclude that the above definition of symbols does neither depend on the chosen uniformly locally finite covering of $M$ through normal coordinate balls, nor on the subordinate partition of unity (as long as the functions $\{\varphi_i\}$ have uniformly bounded derivatives), nor on the synchronous framings of $E$ and $F$.

If all the choices above are fixed, we immediately see from the definition of uniform pseudodifferential operators that $P \in \UPsiDO^k(E,F)$ has a symbol $p \in \Symb^k(E,F)$. Analogously as in the case of compact manifolds,\footnote{see, e.g., \cite[Theorem III.§3.19]{lawson_michelsohn}} we may show that if we make other choices for the coordinate charts, subordinate partition of unity and synchronous framings, the symbol $p$ of $P$ changes by an element of $\Symb^{k-1}(E,F)$. So $P$ has a well-defined principal symbol class $[p] \in \Symb^k(E,F) / \Symb^{k-1}(E,F) =: \Symb^{k-[1]}(E,F)$.

\begin{defn}[Elliptic symbols]\label{defn:elliptic_symbols}
Let $p \in \Symb^k(E,F)$. Recall that $p$ is a section of the bundle $\Hom(\pi^\ast E, \pi^\ast F)$ over $T^\ast M$. We will call $p$ \emph{elliptic}, if there is an $R > 0$ such that $p|_{| \xi | > R}$\footnote{This notation means the following: we restrict $p$ to the bundle $\Hom(\pi^\ast E, \pi^\ast F)$ over the space $\{(x,\xi) \in T^\ast M \ | \ |\xi| > R\} \subset T^\ast M$.} is invertible and this inverse $p^{-1}$ satisfies the Inequality \eqref{eq:symbol_uniformity} for $\alpha, \beta = 0$ and order $-k$ (and of course only for $|\xi| > R$ since only there the inverse is defined). Note that as in the compact case it follows that $p^{-1}$ satisfies the Inequality \eqref{eq:symbol_uniformity} for all multi-indices $\alpha$, $\beta$.
\end{defn}

The proof of the following lemma is straight-forward.

\begin{lem}\label{lem:ellipticity_independent_of_representative}
If $p \in \Symb^k(E,F)$ is elliptic, then every other representative $p^\prime$ of the class $[p] \in \Symb^{k-[1]}(E,F)$ is also elliptic.
\end{lem}

Due to the above lemma we are now able to define what it means for a pseudodifferential operator to be elliptic:

\begin{defn}[Elliptic $\UPsiDO$s]\label{defn:elliptic_operator}
Let $P \in \UPsiDO^k(E,F)$. We will call $P$ \emph{elliptic}, if its principal symbol $\sigma(P)$ is elliptic.
\end{defn}

The importance of elliptic operators lies in the fact that they admit an inverse modulo operators of order $-\infty$. We may prove this analogously as in the case of pseudodifferential operators defined over a compact manifold. See also \cite[Theorem 3.3]{kordyukov} where Kordyukov proves the existence of parametrices for his class of pseudodifferential operators (which coincides with our class with the additional requirement that the operators must have finite propagation).

\begin{thm}[Existence of parametrices]
Let $P \in \UPsiDO^k(E,F)$ be elliptic.

Then there exists an operator $Q \in\UPsiDO^{-k}(F, E)$ such that
\[PQ = \id - S_1 \text{ and }QP = \id - S_2,\]
where $S_1 \in \UPsiDO^{-\infty}(F)$ and $S_2 \in \UPsiDO^{-\infty}(E)$.
\end{thm}

Using parametrices, we can prove a lot of the important properties of elliptic operators, e.g., \emph{elliptic regularity} (which is a converse to Proposition \ref{prop:Pu_smooth_if_u_smooth} and a proof of it may be found in, e.g. \cite[Theorem III.§4.5]{lawson_michelsohn}):

\begin{thm}[Elliptic regularity]\label{thm:elliptic_regularity}
Let $P \in \UPsiDO^k(E,F)$ be elliptic and let furthermore $u \in H^s(E)$ for some $s \in \IZ$.

Then, if $Pu$ is smooth on an open subset $U \subset M$, $u$ is already smooth on $U$. Furthermore, for $k > 0$: if $Pu = \lambda u$ on $U$ for some $\lambda \in \IC$, then $u$ is smooth on $U$.
\end{thm}

Later we will also need the following \emph{fundamental elliptic estimate} (the proof from \cite[Theorem III.§5.2(iii)]{lawson_michelsohn} generalizes directly):

\begin{thm}[Fundamental elliptic estimate]\label{thm:elliptic_estimate}
Let $P \in \UPsiDO^k(E,F)$ be elliptic. Then for each $s \in \IZ$ there is a constant $C_s > 0$ such that
\[\|u\|_{H^s(E)} \le C_s\big(\|u\|_{H^{s-k}(E)} + \|Pu\|_{H^{s-k}(F)}\big)\]
for all $u \in H^s(E)$.
\end{thm}

Another implication of ellipticity is that symmetric\footnote{This means that we have $\langle Pu, v\rangle_{L^2(E)} = \langle u, Pv\rangle_{L^2(E)}$ for all $u,v \in C_c^\infty(E)$.}, elliptic uniform pseudodifferential operators of positive order are essentially self-adjoint\footnote{Recall that a symmetric, unbounded operator is called \emph{essentially self-adjoint}, if its closure is a self-adjoint operator.}. We need this since we will have to consider functions of uniform pseudodifferential operators. But first we will show that a symmetric and elliptic operator is also symmetric as an operator on Sobolev spaces.

\begin{lem}\label{lem:symmetric_on_Sobolev}
Let $P \in \UPsiDO^k(E)$ with $k \ge 1$ be symmetric on $L^2(E)$ and elliptic. Then $P$ is also symmetric on the Sobolev spaces $H^{lk}(E)$ for $l \in \IZ$, where we use on $H^{lk}(E)$ the scalar product as described in the proof.
\end{lem}

\begin{proof}
Due to the fundamental elliptic estimate the norm $\|u\|_{H^0} + \|Pu\|_{H^0}$ (note that $H^0(E) = L^2(E)$ by definition) on $H^k(E)$ is equivalent to the usual\footnote{We have of course possible choices here, e.g., the global norm \eqref{eq:sobolev_norm} or the local definition \eqref{eq:sobolev_norm_local}, but they are all equivalent to each other since $M$ and $E$ have bounded geometry.} norm $\|u\|_{H^k}$ on it. Now $\|u\|_{H^0} + \|Pu\|_{H^0}$ is equivalent to $\big( \|u\|^2_{H^0} + \|Pu\|^2_{H^0} \big)^{1/2}$ which is induced by the scalar product
\[\langle u, v \rangle_{H^k, P} := \langle u, v \rangle_{H^0} + \langle Pu, Pv \rangle_{H^0}.\]
Since $P$ is symmetric for the $H^0$-scalar product, we immediately see that it is also symmetric for this particular scalar product $\langle -, - \rangle_{H^k,P}$ on $H^k(E)$.

To extend to the Sobolev spaces $H^{lk}(E)$ for $l > 0$ we repeatedly invoke the above arguments, e.g., on $H^{2k}(E)$ we have the equivalent norm $\big( \|u\|^2_{H^k, P} + \|Pu\|^2_{H^k, P} \big)^{1/2}$ (again due to the fundamental elliptic estimate) which is induced by the scalar product $\langle u, v \rangle_{H^k,P} + \langle Pu, Pv \rangle_{H^k,P}$ and now we may use that we already know that $P$ is symmetric with respect to $\langle -, - \rangle_{H^k,P}$.

Finally, for $H^{lk}(E)$ for $l < 0$ we use the fact that they are the dual spaces to $H^{-lk}(E)$ where we know that $P$ is symmetric, i.e., we equip $H^{lk}(E)$ for $l < 0$ with the scalar product induced from the duality: $\langle u, v \rangle_{H^{lk},P} := \langle u^\prime, v^\prime\rangle_{H^{-lk},P}$, where $u^\prime, v^\prime \in H^{-lk}(E)$ are the dual vectors to $u,v \in H^{lk}(E)$ (note that the induced norm on $H^{lk}(E)$ is exactly the operator norm if we regard $H^{lk}(E)$ as the dual space of $H^{-lk}(E)$).
\end{proof}

Now we get to the proof that elliptic and symmetric operators are essentially self-adjoint. Note that if we work with differential operators $D$ of first order on open manifolds we do not need ellipticity for this result to hold, but weaker conditions suffice, e.g., that the symbol $\sigma_D$ of $D$ satisfies $\sup_{x \in M, \|\xi\| = 1} \|\sigma_D(x, \xi)\| < \infty$ (by the way, this condition is incorporated in our definition of uniform pseudodifferential operators by the uniformity condition). But if we want essential self-adjointness of higher order operators, we have to assume stronger conditions (see the counterexample \cite{MO_elliptic_essentially_self_adjoint}).

Note that the following proposition is well-known in the case $l=0$, see Shubin \cite{shubin}. But for us it will be of crucial importance in the next Subsection~\ref{sec:functions_of_PDOs} (see the proof of Lemma~\ref{lem:exp(itP)_quasilocal}) that we also have the statement for all the other cases $l \not= 0$. Furthermore, note that in order for the next proposition to make sense at all we have to invoke the above Lemma~\ref{lem:symmetric_on_Sobolev}.

\begin{prop}[Essential self-adjointness]\label{prop:elliptic_PDO_essentially_self-adjoint}
Let $P \in \UPsiDO^k(E)$ with $k \ge 1$ be elliptic and symmetric. Then the unbounded operator $P\colon H^{lk}(E) \to H^{lk}(E)$ is essentially self-adjoint for all $l \in \IZ$, where we equip these Sobolev spaces with the scalar products as described in the proof of the above Lemma \ref{lem:symmetric_on_Sobolev}.
\end{prop}

\begin{proof}
This proof is an adapted version of the proof of this statement for compact manifolds from \cite{MO_elliptic_essentially_self_adjoint}.

We will use the following sufficient condition for essential self-adjointness: if we have a symmetric and densely defined operator $T$ such that $\kernel (T^\ast \pm i) = \{0\}$, then the closure $\overline{T}$ of $T$ is self-adjoint and is the unique self-adjoint extension of $T$.

So let $u \in \kernel (P^\ast \pm i) \subset H^{lk}(E)$, i.e., $P^\ast u = \pm i u$. From elliptic regularity we get that $u$ is smooth and using the fundamental elliptic estimate for $P^\ast$\footnote{Note that $P^\ast$ is elliptic if and only if $P$ is.} we can then conclude $\|u\|_{H^{k+lk}} \le C_{k+lk}\big(\|u\|_{H^{lk}} + \|P^\ast u\|_{H^{lk}}\big) = 2 C_{k+lk} \|u\|_{H^{lk}} < \infty$, i.e., $u \in H^{k+lk}(E)$. Repeating this argument gives us $u \in H^\infty(E)$, i.e., $u$ lies in the domain of $P$ itself and is therefore an eigenvector of it to the eigenvalue $\pm i$. But since $P$ is symmetric we must have $u = 0$. This shows $\kernel (P^\ast \pm i) = \{0\}$ and therefore $P$ is essentially self-adjoint.
\end{proof}

\section{Functions of symmetric, elliptic operators}\label{sec:functions_of_PDOs}

Let $P \in \UPsiDO^k(E)$ be a symmetric and elliptic uniform pseudodifferential operator of positive order $k \ge 1$. By Proposition \ref{prop:elliptic_PDO_essentially_self-adjoint} we know that $P\colon L^2(E) \to L^2(E)$ is essentially self-adjoint. So, if $f$ is a Borel function defined on the spectrum of $P$, the operator $f(P)$ is defined by the functional calculus. In this whole section $P$ will denote such an operator, i.e., a symmetric and elliptic one of positive order.

Given such a uniform pseudodifferential operator $P$, we will later show that it defines naturally a class in uniform $K$-homology. For this we will have to consider $\chi(P)$, where $\chi$ is a so-called normalizing function, and we will have to show that $\chi(P)$ is uniformly pseudolocal and $\chi(P)^2 - 1$ is uniformly locally compact. For this we will need the analysis done in this section.

If $f$ is a Schwartz function, we have the formula $f(P) = \frac{1}{\sqrt{2\pi}}\int_\IR \hat{f}(t) e^{itP} dt$, where $\hat{f}$ is the Fourier transform of $f$. In the case that $P = D$ is an elliptic, first-order differential operator and its symbol satisfies $\sup_{x \in M, \|\xi\| = 1} \|\sigma_D(x, \xi)\| < \infty$, the operator $e^{itD}$ has finite propagation (a proof of this may be found in, e.g., \cite[Proposition 10.3.1]{higson_roe}) from which (exploiting the above formula for $f(D)$) we may deduce the needed properties of $\chi(P)$ and $\chi(P)^2 - 1$. But this is no longer the case for a general elliptic pseudodifferential operator $P$ and therefore the analysis that we have to do here in this general case is much more sophisticated.

Note that the restriction to operators of order $k \ge 1$ in this section is no restriction on the fact that symmetric and elliptic uniform pseudodifferential operators define uniform $K$-homology classes. In fact, if $P$ has order $k \le 0$, then we know from Proposition \ref{prop:PDO_order_0_l-uniformly-pseudolocal} that $P$ is uniformly pseudolocal, i.e., there is no need to form the expression $\chi(P)$ in order for $P$ to define a uniform $K$-homology class.

We start with the following crucial technical lemma which is a generalization of the fact that $e^{itD}$ has finite propagation to pseudodifferential operators. Note that we do not have to assume something like $\sup_{x \in M, \|\xi\| = 1} \|\sigma_D(x, \xi)\| < \infty$ that we had to for first-order differential operators, since such an assumption is subsumed in the uniformity condition that we have in the definition of pseudodifferential operators.

\begin{lem}\label{lem:exp(itP)_quasilocal}
Let $P \in \UPsiDO^{k\ge 1}(E)$ be symmetric and elliptic. Then the operator $e^{itP}$ is a quasilocal operator $H^{l}(E) \to H^{l-(k-1)}(E)$ for all $l \in \IR$ and $t \in \IR$.
\end{lem}

\begin{proof}
This proof is inspired by the proof of \cite[Theorem 3.1]{mcintosh_morris}.

We will need the following two facts:
\begin{enumerate}
\item $\|e^{itP}\|_{l,l} = 1$ for all $l \in \IR$, where $\|-\|_{l,l}$ denotes the operator norm of operators $H^{l}(E) \to H^{l}(E)$ and
\item there exists $\kappa > 0$ such that $\|[\eta, P]\|_{s,s-(k-1)} \le \kappa \cdot \sum_{j=1}^N \|\nabla^j \eta\|_\infty$ for all smooth $\eta \in C_b^\infty(M)$, where $N$ does not depend on $\eta$.
\end{enumerate}

The first point above holds since $e^{itP}$ is a unitary operator $H^{lk}(E) \to H^{lk}(E)$ with $l\in \IZ$ by using Proposition \ref{prop:elliptic_PDO_essentially_self-adjoint}, and by interpolation between the different Sobolev exponents we get the needed norm estimate on any $H^l(E)$ with $l \in \IR$, i.e., not only for integer multiples of $k$.

The second point above is due to the facts that by Proposition \ref{prop:PsiDOs_filtered_algebra} the commutator $[\eta, P]$ is a pseudodifferential operator of order $k-1$ (recall that smooth functions with bounded derivatives are operators of order $0$) and due to Remark \ref{rem:bound_operator_norm_PDO} (where we have to recall the formula how to compute the symbol of the composition of two pseudodifferential operators from, e.g., \cite[Theorem III.§3.10]{lawson_michelsohn}).

Let $L \subset M$ and let $u \in H^{l}(E)$ be supported within $L$. Furthermore, we choose an $R > 0$ and a smooth, real-valued function $\eta$ with $\eta \equiv 1$ on $L$, $\eta \equiv 0$ on $M - B_{R+1}(L)$ and the first $N$ derivatives of $\eta$ (for $N$ as above) bounded from above by $C/R$ for a constant $C$ which does not depend on $u,L,R,\eta$. Concretely, one can construct $\eta$ by mollifying the function $\eta_0(x) := \max\big\{0,1-d(x,B_{1/2}(L))/R\big\}$ with a uniform collection of local mollifiers that are supported in balls of radius less than $1/2$ and with midpoints in the region $B_{R+1/2}(L) - B_{1/2}(L)$. If we denote a local mollifier by $\psi$, then we have for the Lipschitz constant the estimate
\[\operatorname{Lip}\!\big(D^\alpha(\eta_0 \ast \psi)\big) = \operatorname{Lip}( \eta_0 \ast D^\alpha \psi) \le \operatorname{Lip}(\eta_0) \cdot \|D^\alpha \psi\|_{L^1} = 1/R \cdot \|D^\alpha \psi\|_{L^1}\]
from which the needed property on the derivatives of $\eta$ follows. Note that we only need to do this proof for large $R$, i.e., we do the arguments here only for $R$ bigger than, say, the injectivity radius of $M$. This means that the derivatives of the local mollifiers that we use do not explode since there is now a lower bound on the size of the coordinate charts in which we apply our mollifiers.

For all $v \in H^{l-(k-1)}(E)$ that are supported in $M - B_{R+1}(L)$ we have
\begin{align*}
\langle e^{itP} u, v\rangle_{H^{l-(k-1)}} & = \langle e^{itP} \eta u, v\rangle_{H^{l-(k-1)}} - \langle e^{itP} u, \eta v\rangle_{H^{l-(k-1)}}\\
& = \langle [e^{itP},\eta] u, v\rangle_{H^{l-(k-1)}},
\end{align*}
i.e., $|\langle e^{itP} u, v\rangle_{H^{l-(k-1)}}| \le \|[e^{itP},\eta]\|_{l,l-(k-1)} \cdot \|u\|_{H^{l}} \cdot \|v\|_{H^{l-(k-1)}}$ and it remains to give an estimate for $\|[e^{itP},\eta]\|_{l,l-(k-1)}$: we have (the expressions are to be considered point-wise, i.e., after application to a fixed vector $v$)
\begin{align*}
[e^{itP}, \eta] & = \int_0^1 \tfrac{d}{dx} \big( e^{ixtP} \eta e^{i(1-x)tP} \big) dx\\
& = -it \int_0^1 e^{ixtP} [\eta,P] e^{i(1-x)tP} dx
\end{align*}
which gives by factorizing the integrand as
\[H^{l}(E) \stackrel{e^{i(1-x)tP}}\longrightarrow H^{l}(E) \stackrel{[\eta, P]}\longrightarrow H^{l-(k-1)}(E) \stackrel{e^{ixtP}}\longrightarrow H^{l-(k-1)}(E)\]
the estimate
\[\|[e^{itP},\eta]\|_{l,l-(k-1)} \le |t| \int_0^1 \|[\eta, P]\|_{l,l-(k-1)} dx \le |t| \cdot \kappa \cdot \sum_{j=1}^N \|\nabla^j \eta\|_\infty.\]
Since $\|\nabla^j \eta\|_\infty < C/R$ for all $1 \le j \le N$, we have shown
\begin{equation}
\label{eq:dominating_function_exp(itP)}
|\langle e^{itP} u, v\rangle_{H^{l-(k-1)}}| < \frac{|t| \kappa N C}{R} \cdot \|u\|_{H^{l}} \cdot \|v\|_{H^{l-(k-1)}}
\end{equation}
for all $u$ supported in $L$ and all $v$ in $M-B_{R+1}(L)$. Because $R > 0$ and $l \in \IR$, $t \in\IR$ were arbitrary, the claim that $e^{itP}$ is a quasilocal operator $H^{l}(E) \to H^{l-(k-1)}(E)$ for all $l \in \IR$ and $t \in \IR$ follows.
\end{proof}

\begin{cor}[cf.~{\cite[Lemma 1.1 in Chapter XII.§1]{taylor_pseudodifferential_operators}}]
\label{cor:lth_derivative_integrable_defines_quasilocal_operator}
Let $q(t)$ be a function on $\IR$ such that for an $n \in \IN_0$ the functions $q(t)|t|$, $q^\prime(t)|t|$,  $\ldots$, $q^{(n)}(t)|t|$ are integrable, i.e., belong to $L^1(\IR)$.

Then the operator defined by $\int_\IR q(t) e^{itP} dt$ is for all values $l \in \IR$ a quasilocal operator $H^{l-nk+k-1}(E) \to H^{l}(E)$, i.e., is of order $-nk + k - 1$.
\end{cor}

\begin{proof}
Let $Q \in \UPsiDO^{-k}(E)$ be a parametrix for $P$, i.e., $PQ = \id - S_1$ and $QP = \id - S_2$, where $S_1, S_2 \in \UPsiDO^{-\infty}(E)$. Integration by parts $n$ times yields:
\begin{equation}
\label{eq:formula_integration_by_parts_quasilocal}
(i Q)^n \int_\IR q^{(n)}(t) e^{itP} dt = (i Q)^n (-i P)^n \int_\IR q(t) e^{itP} dt = (\id - S_2)^n \int_\IR q(t)e^{itP} dt.
\end{equation}
Since $q(t)|t|$ and $q^{(n)}(t)|t|$ are integrable and due to the Estimate \eqref{eq:dominating_function_exp(itP)}, we conclude with Lemma \ref{lem:exp(itP)_quasilocal} that both integrals $\int_\IR q(t)e^{itP} dt$ and $\int_\IR q^{(n)}(t)e^{itP} dt$ define quasilocal operators of order $k-1$ on $H^{l}(E)$. Note that for $\int_\IR q(t)e^{itP} dt$ this is just a first result which we will need now in order to show that the order of this operator is in fact lower.

Now $(\id - S_2)^n = \id + \sum_{j=1}^n \binom{n}{j}(-S_2)^j$ and the sum is a quasilocal smoothing operator because $S_2$ is one. Since the composition of quasilocal operators is again a quasilocal operator (see \cite[Proposition 5.2]{roe_index_1}), we conclude that the second summand $R$ of
\begin{equation}
\label{eq:formula_integration_by_parts_quasilocal_2}
(\id - S_2)^n \int_\IR q(t)e^{itP} dt = \int_\IR q(t)e^{itP} dt + \underbrace{\sum_{j=1}^n \binom{n}{j}(-S_2)^j \int_\IR q(t)e^{itP} dt}_{=: R}
\end{equation}
is also a quasilocal smoothing operator. Now Equations \eqref{eq:formula_integration_by_parts_quasilocal} and \eqref{eq:formula_integration_by_parts_quasilocal_2} together yield
\[\int_\IR q(t)e^{itP} dt = (i Q)^n \int_\IR q^{(n)}(t) e^{itP} dt - R,\]
from which the claim follows.
\end{proof}

Recall that if $f$ is a Schwartz function, then the operator $f(P)$ is given by
\begin{equation}\label{eq:schwartz_function_of_PDO}
f(P) = \frac{1}{\sqrt{2\pi}}\int_\IR \hat{f}(t) e^{itP} dt,
\end{equation}
where $\hat{f}$ is the Fourier transform of $f$. Since $\hat{f}$ is also a Schwartz function, it satisfies the assumption in Corollary \ref{cor:lth_derivative_integrable_defines_quasilocal_operator} for all $n \in \IN_0$, i.e., $f(P)$ is a quasilocal smoothing operator. Applying this argument to the adjoint operator $f(P)^\ast = \overline{f}(P)$, we get with Lemma \ref{lem:PDO_-infinity_equal_quasilocal_smoothing} our next corollary:

\begin{cor}\label{cor:schwartz_function_of_PDO_quasilocal_smoothing}
If $f$ is a Schwartz function, then $f(P) \in \UPsiDO^{-\infty}(E)$.
\end{cor}

Recall from \cite[Lemma 4.2]{spakula_uniform_k_homology} that the uniformly pseudolocal operators form a $C^\ast$-algebra and that the uniformly locally compact operators form a closed, two-sided $^\ast$-ideal in there. Since Schwartz functions are dense in $C_0(\IR)$ and quasilocal smoothing operators are uniformly locally compact (Corollary \ref{cor:quasilocal_neg_order_uniformly_locally_compact}), we get with the above corollary that $g(P)$ is uniformly locally compact if $g \in C_0(\IR)$.

\begin{cor}\label{cor:g(P)_uniformly_locally_compact_g_vanishing_at_infinity}
Let $g \in C_0(\IR)$. Then $g(P)$ is uniformly locally compact.
\end{cor}

Now we turn our attention to functions which are more general than Schwartz functions. To be concrete, we consider functions of the following type:

\begin{defn}[Symbols on $\IR$]\label{defn:symbols_on_R}
For arbitrary $m \in \IZ$ we define
\[\mathcal{S}^m(\IR) := \{f \in C^\infty(\IR) \ | \ |f^{(n)}(x)| < C_n(1 + |x|)^{m-n} \text{ for all } n \in \IN_0\}.\]

Note that we have $\mathcal{S}(\IR) = \bigcap_m \mathcal{S}^m(\IR)$, where $\mathcal{S}(\IR)$ denotes the Schwartz space.
\end{defn}

Let us state now the generalization of \cite[Theorem 5.5]{roe_index_1} from operators of Dirac type to uniform pseudodifferential operators:

\begin{prop}[cf.~{\cite[Theorem 5.5]{roe_index_1}}]\label{prop:f(P)_quasilocal_of_symbol_order}
Let $f \in \mathcal{S}^m(\IR)$ with $m \le 0$. Then for all $l \in \IR$ the operator $f(P)$ is a quasilocal operator of order $mk+k-1$, i.e., is an operator $f(P) \colon H^{l}(E) \to H^{l-(mk+k-1)}(E)$.
\end{prop}

\begin{proof}
The proof is analogous to Roe's proof of \cite[Theorem 5.5]{roe_index_1}, but more technical. First let us note that $f(P)$ is a bounded operator of order $mk$. To see this note that $(1+|x|)^{-m} \cdot f(x)$ is a bounded function and therefore $(1+|P|)^{-m} \circ f(P)$ is a bounded operator of order $0$. Combining the fact that $(1+|P|)^{-m}$ is an operator of order $-mk$ together with the fundamental elliptic estimate from Theorem \ref{thm:elliptic_estimate} we get the result that $f(P)$ is bounded of order $mk$.

Now we want not only boundedness of $f(P)$ but also that it is quasilocal. Roe uses in his proof of \cite[Theorem 5.5]{roe_index_1} the fact that $e^{itD}$ has propagation $|t|$ for $D$ a Dirac operator. But for pseudodifferential operators the best that we have is our Lemma \ref{lem:exp(itP)_quasilocal} and that's the reason why we loose $k-1$ orders for the statement that $f(P)$ is quasilocal. The rest of our proof is analogous to Roe's proof.
\end{proof}

At last, let us turn our attention to a result regarding differences $\psi(P) - \psi(P^\prime)$ of operators defined via functional calculus. We will need the following proposition in the proof of the proposition where we show that symmetric, elliptic uniform pseudodifferential operators with the same symbol define the same uniform $K$-homology class.

\begin{prop}[{\cite[Proposition 10.3.7]{higson_roe}}\footnote{The cited proposition requires additionally a common invariant domain for $P$ and $P^\prime$. In our case here this domain is given by, e.g., $H^\infty(E)$.}]\label{prop:norm_estimate_difference_func_calc}
Let $\psi$ be a bounded Borel function whose distributional Fourier transform $\hat{\psi}$ is such that the product $s\hat{\psi}(s)$ is in $L^1(\IR)$.

If $P$ and $P^\prime$ are symmetric and elliptic uniform pseudodifferential operators of positive order $k \ge 1$ such that their difference $P - P^\prime$ has order $q$, then we have for all $l \in \IR$
\[\|\psi(P) - \psi(P^\prime)\|_{l,l-q} \le C_\psi \cdot \|P - P^\prime\|_{l,l-q},\]
where the constant $C_\psi = \frac{1}{2\pi} \int |s \hat{\psi}(s)| ds$ does not depend on the operators.
\end{prop}

\begin{proof}
We first assume that $\hat{\psi}$ is compactly supported and  $s\hat{\psi}(s)$ a smooth function. Then we use the result \cite[Proposition 10.3.5]{higson_roe}\footnote{Though stated there only for differential operators, its proof also works word-for-word for pseudodifferential ones.}, which is a generalization of Equation~\ref{eq:schwartz_function_of_PDO} to more general functions than Schwartz functions, and get
\[\Big\langle \big( \psi(P) - \psi(P^\prime) \big) u, v \Big\rangle_{H^{l-q}} = \frac{1}{2 \pi} \int \left\langle \big( e^{isP} - e^{isP^\prime} \big) u, v \right\rangle_{H^{l-q}} \cdot \hat{\psi}(s) ds,\]
for all $u,v \in C_c^\infty(E)$. From the Fundamental Theorem of Calculus we get
\[\left\langle \big( e^{isP} - e^{isP^\prime} \big) u, v \right\rangle_{H^{l-q}} = i \cdot \int_0^s \left\langle \big( e^{itP} (P - P^\prime) e^{i(s-t)P^\prime} \big) u, v \right\rangle_{H^{l-q}} dt\]
and therefore
\[\left| \left\langle \big( e^{isP} - e^{isP^\prime} \big) u, v \right\rangle_{H^{l-q}}\right| \le s \cdot \|P - P^\prime\|_{l,l-q} \cdot \|u\|_{l} \cdot \|v\|_{l-q}.\]
Putting it all together, we get
\[\left| \Big\langle \big( \psi(P) - \psi(P^\prime) \big) u, v \Big\rangle_{H^{l-q}} \right| \le C_\psi \cdot \|P - P^\prime\|_{l,l-q} \cdot \|u\|_{l} \cdot \|v\|_{l-q}.\]

Now the general claim follows from an approximation argument analogous to the one at the end of the proof of \cite[Proposition 10.3.5]{higson_roe}.
\end{proof}

\section{Review of uniform \texorpdfstring{$K$}{K}-homology}

Let us first recall briefly the notion of multigraded Hilbert spaces. They arise as $L^2$-spaces of vector bundles on which Clifford algebras act.

A \emph{graded Hilbert space} is a Hilbert space $H$ with a decomposition $H = H^+ \oplus H^-$ into closed, orthogonal subspaces. This is equivalent to the existence of a \emph{grading operator} $\epsilon$ such that its $\pm 1$-eigenspaces are exactly $H^\pm$ and such that $\epsilon$ is a selfadjoint unitary.

If $H$ is a graded space, then its \emph{opposite} is the graded space $H^\op$ whose underlying vector space is $H$, but with the reversed grading, i.e., $(H^\op)^+ = H^-$ and $(H^\op)^- = H^+$. This is equivalent to $\epsilon_{H^\op} = -\epsilon_H$.

An operator on a graded space $H$ is called \emph{even} if it maps $H^\pm$ again to $H^\pm$, and it is called \emph{odd} if it maps $H^\pm$ to $H^\mp$. Equivalently, an operator is even if it commutes with the grading operator $\epsilon$ of $H$, and it is odd if it anti-commutes with it.

\begin{defn}[Multigraded Hilbert spaces]
Let $p \in \IN_0$. A \emph{$p$-multigraded Hilbert space} is a graded Hilbert space which is equipped with $p$ odd unitary operators $\epsilon_1, \ldots, \epsilon_p$ such that $\epsilon_i \epsilon_j + \epsilon_j \epsilon_i = 0$ for $i \not= j$, and $\epsilon_j^2 = -1$ for all $j$.
\end{defn}

Note that a $0$-multigraded Hilbert space is just a graded Hilbert space. We make the convention that a $(-1)$-multigraded Hilbert space is an ungraded one.

\begin{defn}[Multigraded operators]
Let $H$ be a $p$-multigraded Hilbert space. Then an operator on $H$ will be called \emph{multigraded}, if it commutes with the multigrading operators $\epsilon_1, \ldots, \epsilon_p$ of $H$.
\end{defn}

To define uniform Fredholm modules we will need the following notions. Let us define
\begin{equation*}
\LLip_R(X) := \{ f \in C_c(X) \ | \ f \text{ is }L\text{-Lipschitz}, \diam(\supp f) \le R \text{ and } \|f\|_\infty \le 1\}.
\end{equation*}

\begin{defn}[{\cite[Definition 2.3]{spakula_uniform_k_homology}}]\label{defn:uniform_operators}
Let $T \in \IB(H)$ be an operator on a Hilbert space $H$ and $\rho\colon C_0(X) \to \IB(H)$ a representation.

We say that $T$ is \emph{uniformly locally compact}, if for every $R, L > 0$ the collection
\[\{\rho(f)T, T\rho(f) \ | \ f \in \LLip_R(X)\}\]
is uniformly approximable (see Definition \ref{defn:uniformly_approximable_collection}).

We say that $T$ is \emph{uniformly pseudolocal}, if for every $R, L > 0$ the collection
\[\{[T, \rho(f)] \ | \ f \in \LLip_R(X)\}\]
is uniformly approximable.
\end{defn}

\begin{defn}[Multigraded uniform Fredholm modules, cf.~{\cite[Definition 2.6]{spakula_uniform_k_homology}}]
Let $p \in \IZ_{\ge -1}$. A triple $(H,\rho,T)$ consisting of
\begin{itemize}
\item a separable $p$-multigraded Hilbert space $H$,
\item a representation $\rho\colon C_0(X) \to \IB(H)$ by even, multigraded operators, and
\item an odd multigraded operator $T \in \IB(H)$ such that
\begin{itemize}
\item the operators $T^2 - 1$ and $T - T^\ast$ are uniformly locally compact and
\item the operator $T$ itself is uniformly pseudolocal
\end{itemize}
\end{itemize}
is called a \emph{$p$-multigraded uniform Fredholm module over $X$}.
\end{defn}

\begin{example}[{\cite[Theorem 3.1]{spakula_uniform_k_homology}}]
\Spakula showed that the usual Fredholm module arising from a generalized Dirac operator is uniform if we assume bounded geometry: if $D$ is a generalized Dirac operator acting on a Dirac bundle $S$ of bounded geometry over a manifold $M$ of bounded geometry, then the triple $(L^2(S), \rho, \chi(D))$, where $\rho$ is the representation of $C_0(M)$ on $L^2(S)$ by multiplication operators and $\chi$ is a normalizing function (see Definition \ref{defn:normalizing_function}), is a uniform Fredholm module. It is multigraded if the Dirac bundle $S$ has an action of a Clifford algebra.
\end{example}

A collection $(H, \rho, T_t)$ of uniform Fredholm modules is called an \emph{operator homotopy} if $t \mapsto T_t \in \IB(H)$ is norm continuous.

\begin{defn}[Uniform $K$-homology, {\cite[Definition 2.13]{spakula_uniform_k_homology}}]
We define the \emph{uniform $K$-homology group $K_{p}^u(X)$} of a locally compact and separable metric space $X$ to be the abelian group generated by unitary equivalence classes of $p$-multigraded uniform Fredholm modules with the relations:
\begin{itemize}
\item if $x$ and $y$ are operator homotopic, then $[x] = [y]$, and
\item $[x] + [y] = [x \oplus y]$,
\end{itemize}
where $x$ and $y$ are $p$-multigraded uniform Fredholm modules.
\end{defn}

To prove that the uniform $K$-homology class of an elliptic uniform pseudodifferential operator only depends on the principal symbol of the operator we will need the result that weakly homotopic Fredholm modules define the same $K$-homology class.

\begin{defn}[Weak homotopies]\label{defn:weak_homotopy}
Let a time-parametrized family of uniform Fredholm modules $(H, \rho_t, T_t)$ for $t \in [0,1]$ satisfy the following properties:
\begin{itemize}
\item the family $\rho_t$ is pointwise strong-$^\ast$ operator continuous, i.e., for all $f \in C_0(X)$ we get a path $\rho_t(f)$ in $\IB(H)$ that is continuous in the strong-$^\ast$ operator topology\footnote{Recall that if $H$ is a Hilbert space, then the \emph{strong-$^\ast$ operator topology} on $\IB(H)$ is generated by the family of seminorms $p_v(T) := \|Tv\| + \|T^\ast v\|$ for all $v \in H$, where $T \in \IB(H)$.},
\item the family $T_t$ is continuous in the strong-$^\ast$ operator topology on $\IB(H)$, i.e., for all $v \in H$ we get norm continuous paths $T_t(v)$ and $T_t^\ast(v)$ in $H$, and
\item for all $f \in C_0(X)$ the families of compact operators $[T_t, \rho_t(f)]$, $(T_t^2 - 1)\rho_t(f)$ and $(T_t-T_t^\ast)\rho_t(f)$ are norm continuous.
\end{itemize}
Then we call it a \emph{weak homotopy} between $(H, \rho_0, T_0)$ and $(H, \rho_1, T_1)$.
\end{defn}

\begin{thm}[{\cite[Theorem 3.30]{engel_indices_UPDO}}]\label{thm:weak_homotopy_equivalence_K_hom}
Let $(H, \rho_0, T_0)$ and $(H, \rho_1, T_1)$ be weakly homotopic uniform Fredholm modules over a manifold of bounded geometry.\footnote{The result hold more general spaces, but in this paper we need it only for manifolds.}

Then they define the same uniform $K$-homology class.
\end{thm}

\section{\texorpdfstring{$K$}{K}-homology classes of uniform elliptic operators}\label{sec:homology_classes_of_PDOs}

We will show that symmetric, elliptic uniform pseudodifferential operators of positive order naturally define classes in uniform $K$-homology. This result is a crucial generalization of \cite[Theorem 3.1]{spakula_uniform_k_homology}, where this statement is proved for generalized Dirac operators.

First we need a definition and then we will plunge right into the main result:

\begin{defn}[Normalizing functions]\label{defn:normalizing_function}
A smooth function $\chi\colon \IR \to [-1, 1]$ with
\begin{itemize}
\item $\chi$ is odd, i.e., $\chi(x) = -\chi(-x)$ for all $x \in \IR$,
\item $\chi(x) > 0$ for all $x > 0$, and
\item $\chi(x) \to \pm 1$ for $x \to \pm \infty$
\end{itemize}
is called a \emph{normalizing function}.
\end{defn}

\begin{thm}\label{thm:elliptic_symmetric_PDO_defines_uniform_Fredholm_module}
Let $M$ be a manifold of bounded geometry and without boundary, $E \to M$ be a $p$-multigraded vector bundle of bounded geometry, $P \in \UPsiDO^k(E)$ be a symmetric and elliptic uniform pseudodifferential operator on $E$ of positive order $k \ge 1$, and let $P$ be odd and multigraded.

Then $(H, \rho, \chi(P))$ is a $p$-multigraded uniform Fredholm module over $M$, where the Hilbert space is $H := L^2(E)$, the representation $\rho\colon C_0(M) \to \IB(H)$ is by multiplication operators and $\chi$ is a normalizing function. Furthermore, the uniform $K$-homology class $[(H, \rho, \chi(P))] \in K_{p}^u(M)$ does not depend on the choice of $\chi$.
\end{thm}

\begin{proof}
To show that $(H, \rho, \chi(P))$ defines a uniform Fredholm module over $M$ we have to show that $\chi(P)$ is uniformly pseudolocal and that $\chi(P)^2 - 1$ and $\chi(P) - \chi(P)^\ast$ are uniformly locally compact.

Since $\chi$ is real-valued and $P$ is essentially self-adjoint (by Proposition \ref{prop:elliptic_PDO_essentially_self-adjoint}), we have $\chi(P) - \chi(P)^\ast = 0$, i.e., the operator $\chi(P) - \chi(P)^\ast$ is trivially uniformly locally compact. Moreover, since we have $\chi(P)^2 - 1 = (\chi^2 - 1)(P)$ and $\chi^2-1 \in C_0(\IR)$, we conclude with Corollary \ref{cor:g(P)_uniformly_locally_compact_g_vanishing_at_infinity} that $\chi(P)^2 - 1$ is uniformly locally compact.

Because the difference of two normalizing functions is a function from $C_0(\IR)$, we conclude from the same corollary that in order to show that $\chi(P)$ is uniformly pseudolocal, it suffices to show this for one particular normalizing function (and secondly, we get that the class $[(H, \rho, \chi(P))]$ is independent of the concrete choice of $\chi$ since perturbations by uniformly locally compact operators are operator homotopic).

From now on we proceed as in the proof of \cite[Theorem 3.1]{spakula_uniform_k_homology} using the same formulas: we choose the particular normalizing function $\chi(x) := \frac{x}{\sqrt{1+x^2}}$ to prove that $\chi(P)$ is uniformly pseudolocal. We have $\chi(P) = \frac{2}{\pi} \int_0^\infty \frac{P}{1 + \lambda^2 + P^2} d\lambda$ with convergence of the integral in the strong operator topology\footnote{This follows from the equality $\frac{x}{\sqrt{1 + x^2}} = \frac{2}{\pi} \int_0^\infty \frac{x}{1 + \lambda^2 + x^2} d\lambda$ for all $x \in \IR$.} and get then for $f \in \LLip_R(M)$
\begin{equation*}
[\rho(f), \chi(P)] = \frac{2}{\pi} \int_0^\infty \frac{1}{1 + \lambda^2 + P^2} \big( (1+\lambda^2)[\rho(f), P] + P[\rho(f),P]P \big) \frac{1}{1 + \lambda^2 + P^2} d\lambda.
\end{equation*}

Suppose $f \in \LLip_R(M) \cap C_b^\infty(M)$. Then the integral converges in operator norm. To see this, we have to find upper bounds for the operator norms of $\frac{1+\lambda^2}{1+\lambda^2 + P^2} [\rho(f), P] \frac{1}{1+\lambda^2 + P^2}$ and $\frac{P}{1+\lambda^2 + P^2} [\rho(f), P] \frac{P}{1+\lambda^2 + P^2}$, that are integrable with respect to $\lambda$. Recall Definition~\ref{defn:symbols_on_R} of the symbol classes on $\IR$:
\[\mathcal{S}^m(\IR) := \{g \in C^\infty(\IR) \ | \ |g^{(n)}(x)| < C_l(1 + |x|)^{m-n} \text{ for all } n \in \IN_0\}.\]
Since both $\frac{1}{1+\lambda^2 + x^2} \in \mathcal{S}^{-2}(\IR)$ and $\frac{1 + \lambda^2}{1+\lambda^2 + x^2} \in \mathcal{S}^{-2}(\IR)$ (with respect to the variable $x$, i.e., for fixed $\lambda$), the operators $\frac{1}{1+\lambda^2 + P^2}$ and $\frac{1 + \lambda^2}{1+\lambda^2 + P^2}$ are operators of order $-2k$ by the first paragraph of the proof of Proposition \ref{prop:f(P)_quasilocal_of_symbol_order}. So $\frac{1+\lambda^2}{1+\lambda^2 + P^2} [\rho(f), P] \frac{1}{1+\lambda^2 + P^2}$ is an operator of order $-3k-1$ since $[\rho(f),P]$ is of order $k-1$ by Proposition~\ref{prop:PsiDOs_filtered_algebra}. So especially it is a bounded operator, and one can show that there is an integrable upper bound on the operator norm with respect to $\lambda$. The latter can be done by, e.g., using the estimates that Roe derived in his proof of his version of Proposition~\ref{prop:f(P)_quasilocal_of_symbol_order}. Analogously we can treat $\frac{P}{1+\lambda^2 + P^2} [\rho(f), P] \frac{P}{1+\lambda^2 + P^2}$ which is an operator of order $-k-1$.

Furthermore, there exists an $N > 0$ which depends only on an $\varepsilon > 0$, $R = \diam (\supp f)$ and the norms of the derivatives of $f$,\footnote{The dependence on $R$ and on the derivatives of $f$ comes from the operator norm estimate of $[\rho(f), P]$.} such that there are $\lambda_1, \ldots, \lambda_N$ and the above integral is at most $\varepsilon$ away from the sum of the integrands for $\lambda_1, \ldots, \lambda_N$.

Since both $\frac{1}{1+\lambda^2 + x^2} \in \mathcal{S}^{-2}(\IR)$ and $\frac{1 + \lambda^2}{1+\lambda^2 + x^2} \in \mathcal{S}^{-2}(\IR)$ (with respect to the variable $x$, i.e., for fixed $\lambda$), the operators $\frac{1}{1+\lambda^2 + P^2}$ and $\frac{1 + \lambda^2}{1+\lambda^2 + P^2}$ are quasilocal operators of order $-k-1$ by Proposition~\ref{prop:f(P)_quasilocal_of_symbol_order}. This also holds for their adjoints and so, by Corollary~\ref{cor:quasilocal_neg_order_uniformly_locally_compact}, they are uniformly locally compact. The same conclusion applies to the operators $\frac{P}{1+\lambda^2+P^2}$ and $\frac{(1+\lambda^2)P}{1+\lambda^2+P^2}$ which are quasilocal of order $-1$ and hence also uniformly locally compact.

So the first summand
\[\frac{1+\lambda^2}{1+\lambda^2 + P^2} [\rho(f), P] \frac{1}{1+\lambda^2 + P^2}\]
of the integrand is the difference of two compact operators and their approximability by finite rank operators depends only on $R = \diam (\supp f)$ and the Lipschitz constant $L$ of~$f$. An analogous argument applies to the second summand
\[\frac{1}{1+\lambda^2 + P^2}P [\rho(f), P] P \frac{1}{1+\lambda^2 + P^2}\]
of the integrand (note that $\frac{P^2}{1+\lambda^2+P^2}$ is a bounded operator).

So the operator $[\rho(f), \chi(P)]$ is for $f \in \LLip_R(M) \cap C_b^\infty(M)$ compact and its approximability by finite rank operators depends only on $R$, $L$ and the norms of the derivatives of $f$. That this suffices to conclude that the operator is uniformly pseudolocal is exactly Point~5 in Lemma \ref{lem:kasparov_lemma_uniform_approx_manifold}.

To conclude the proof we have to show that $\chi(P)$ is odd and multigraded. But this was already shown in full generality in \cite[Lemma 10.6.2]{higson_roe}.
\end{proof}

We have shown in the above theorem that a symmetric, elliptic uniform pseudodifferential operator naturally defines a class in uniform $K$-homology. Now we will show that this class does only depend on the principal symbol of the pseudodifferential operator. Note that ellipticity of an operator does only depend on its symbol (since it is actually defined that way, see Definition \ref{defn:elliptic_operator}, which is possible due to Lemma \ref{lem:ellipticity_independent_of_representative}), i.e., another pseudodifferential operator with the same symbol is automatically also elliptic.

\begin{prop}\label{prop:same_symbol_same_k_hom_class}
The uniform $K$-homology class of a symmetric and elliptic uniform pseudodifferential operator $P \in \UPsiDO^{k \ge 1}(E)$ does only depend on its principal symbol $\sigma(P)$, i.e., any other such operator $P^\prime$ with the same principal symbol defines the same uniform $K$-homology class.
\end{prop}

\begin{proof}
Consider in $\UPsiDO^k(E)$ the linear path $P_t := (1-t)P + t P^\prime$ of operators. They are all symmetric and, since $\sigma(P) = \sigma(P^\prime)$, they all have the same principal symbol. So they are all elliptic and therefore we get a family of uniform Fredholm modules $(H, \rho, \chi(P_t))$, where we use a fixed normalizing function $\chi$.

Now if the family $\chi(P_t)$ of bounded operators would be norm-continuous, the claim that we get the same uniform $K$-homology classes would follow directly from the relations defining uniform $K$-homology. But it seems that in general it is only possible to conclude the norm continuity of $\chi(P_t)$ if the difference $P - P^\prime$ is a bounded operator,\footnote{see, e.g., \cite[Proposition 10.3.7]{higson_roe}} i.e., if the order $k$ of $P$ is $1$ (since then the order of the difference $P - P^\prime$ would be $0$, i.e., it would define a bounded operator on $L^2(E)$); see Proposition~\ref{prop:norm_estimate_difference_func_calc}.

In the case $k > 1$ we get continuity of $\chi(P_t)$ only in the strong-$^\ast$ operator topology on $\IB(L^2(E))$. This is seen with Proposition \ref{prop:norm_estimate_difference_func_calc},\footnote{An example of a normalizing function $\chi$ fulfilling the prerequisites of Proposition \ref{prop:norm_estimate_difference_func_calc} may be found in, e.g., \cite[Exercise 10.9.3]{higson_roe}.} which implies that the family $t \mapsto \chi(P_t)$ is continuous in the norm topology of operators of degree $k-1$. Therefore, if $v \in L^2(E)$ is an element of the Sobolev space $H^{k-1}(E) \subset L^2(E)$, then $t \mapsto \chi(P_t)(v)$ is norm continuous for the $L^2$-norm. For general $v \in L^2(E)$ we do an approximation argument.

To show that $(H, \rho, \chi(P_0))$ and $(H, \rho, \chi(P_1))$ define the same uniform $K$-homology class we will use Theorem \ref{thm:weak_homotopy_equivalence_K_hom}, i.e., we will show now that the family $(H, \rho, \chi(P_t))$ is a weak homotopy.

The first bullet point of the definition of a weak homotopy is clearly satisfied since our representation $\rho$ is fixed, i.e., does not depend on the time $t$. Moreover, we have already discussed the second bullet point in the paragraph above, so it remains to varify that the third point is satisfied. We will treat here only the case $[\rho(f), \chi(P_t)]$ since the arguments for $\rho(f)(\chi(P_t)^2-1)$ are similar and the case of $\rho(f)(\chi(P_t) - \chi(P_t)^\ast)$ is clear since $\chi(P_t) - \chi(P_t)^\ast = 0$, because $P_t$ is essentially self-adjoint.

So let $\chi$ be the normalizing function $\chi(x) = \frac{x}{\sqrt{1+x^2}}$. This is the one used in the proof of the above Theorem~\ref{thm:elliptic_symmetric_PDO_defines_uniform_Fredholm_module} and we use the integral representation of $[\rho(f), \chi(P_t)]$ derived in that proof. We will only treat the second summand $\frac{1}{1+\lambda^2+P_t^2} P_t[\rho(f),P_t]P_t \frac{1}{1+\lambda^2+P_t^2}$ since it contains two more $P_t$ than the first summand (i.e., it is harder to deal with the second summand than with the first one). We have $\frac{1}{1+\lambda^2 + x^2} \in \mathcal{S}^{-2}(\IR)$ (with respect to $x$) and therefore $\psi(x) := x^{2+\varepsilon} \frac{1}{1+\lambda^2 + x^2} \in \mathcal{S}^{\varepsilon}(\IR)$. So $\psi^\prime(x) \in \mathcal{S}^{\varepsilon-1}(\IR)$ and ${\psi^\prime}{}^\prime(x) \in \mathcal{S}^{\varepsilon-2}(\IR)$, which means that both are $L^2$-integrable if $\varepsilon < 1/2$, i.e., $\psi^\prime(x) \in H^1(\IR)$. Therefore the Fourier transform of $\psi^\prime(x)$ is $L^1$-integrable. But the Fourier transform of $\psi^\prime(x)$ is $s \cdot \widehat{\psi}(s)$, i.e., $\psi$ qualifies for Proposition~\ref{prop:norm_estimate_difference_func_calc} (that $\psi^\prime(x)$ is not bounded is ok, the proposition still works in this case). So $t \mapsto \psi(P_t)$ will be continuous in $\|-\|_{0,-k+1}$-norm. By elliptic regularity this means that $t \mapsto \frac{1}{1+\lambda^2+P_t^2}$ is continuous in $\|-\|_{0,(1+\varepsilon)(k-1)}$-norm. Since $t \mapsto [\rho(f),P_t]$ is continuous in $\|-\|_{0,-k+2}$-norm, we conclude that the whole second summand is continuous in $\|-\|_{0,2\varepsilon(k-1)-k+2}$-norm. If $\varepsilon = 1/2$, then $2\varepsilon(k-1)-k+2 = 1$. Since we have to choose $\varepsilon < 1/2$, we choose it just beneath $1/2$, i.e., so that $2\varepsilon(k-1)-k+2 > 0$. It then follows that $[\rho(f), \chi(P_t)]$ is continuous in operator norm, which concludes this proof.
\end{proof}

\section{Final remarks and open questions}\label{sec_final_remarks}

\subsection*{Quasilocal operators and questions of propagation}

In the definition of uniform pseudodifferential operators we used for the $(-\infty)$-part of them quasilocal smoothing operators. Now the definition of quasi-local operators tempts one to think that such operators might be approximable by finite propagation operators, but this is actually an open problem.

The first results obtained in this direction were by Rabinovich--Roch--Silbermann \cite{RRS}, resp., of Lange--Rabinovich \cite{lange_rabinovich} that on $\IR^n$ every quasi-local operator is approximable by finite propagation operators. This result was recently generalized by {\v{S}}pakula and Tikuisis \cite{st} to all metric spaces with finite decomposition complexity. The only other (partial) result that the author knows is his own \cite{engel_rough} that on spaces of polynomial growth one can approximate operators with a super-polynomially fast decaying dominating function by finite propagation operators.

Currently the main question in this matter is whether one can actually construct a counter-example:

\begin{question}
Does there exist a quasilocal operator which is not approximable by finite propagation operators?
\end{question}

The class of uniform pseudodifferential operators defined in this article is in the following sense connected to the above question: assume that we would have defined our class of operators in such a way that the $(-\infty)$-part would be an operator which is in the \Frechet closure\footnote{That is to say, in the closure with respect to the family of norms $(\|-\|_{-k,l}, \|-^\ast\|_{-k,l})_{k,l \in \IN}$, where $\|-\|_{-k,l}$ denotes the operator norm $H^{-k}(E) \to H^l(E)$.} of the finite propagation smoothing operators. Then the results of Section \ref{sec:uniformity_PDOs} would give a direct connection to the uniform Roe algebra: we would then be able to conclude $\overline{\UPsiDO^{-\infty}(E)} = \overline{\UPsiDO^{-1}(E)} = C_u^\ast(E)$, where $C_u^\ast(E)$ is the uniform Roe algebra of $E$, i.e., the closure of the finite propagation, uniformly locally compact operators on $E$.

If we would do the above, i.e., changing the definition from quasilocal to approximable by finite propagation operators, there would be one piece of information missing that we do have at our disposal by using quasilocal operators: recall that in the analysis of uniform pseudodifferential operators Lemma \ref{lem:exp(itP)_quasilocal} was the main technical ingredient which led, e.g., to Corollary \ref{cor:schwartz_function_of_PDO_quasilocal_smoothing} stating that if $f$ is a Schwartz function, then $f(P) \in \UPsiDO^{-\infty}(E)$ for $P$ a symmetric and elliptic uniform pseudodifferential operator of positive order. But the author does not know whether Lemma \ref{lem:exp(itP)_quasilocal} would also hold for the changed definition, i.e., whether under the conditions of that lemma the operator $e^{itP}$ would be approximable in the needed operator norm by finite propagation operators.

\begin{question}\label{quesnksdio2}
Does Lemma \ref{lem:exp(itP)_quasilocal} specialize to the statement that if the $(-\infty)$-part of $P$ is in the \Frechet closure of the finite propagation smoothing operators, then $e^{itP}$ is approximable by finite propagation operators of order $k$ in the operator norms $\|-\|_{lk,lk-k}$ for all $l \in \IZ$?
\end{question}

Given a generalized Dirac operator $D$, the construction of its rough index class\footnote{The construction of the rough index class is analogous to the construction of the coarse one. A suitable reference is, e.g., \cite[Section 4.3]{roe_coarse_cohomology}.} produces directly a representative of it with finite propagation. The reason for this is that the wave operator $e^{itD}$ has finite propagation.

If we have a symmetric and elliptic uniform pseudodifferential operator $P$, we get a rough index class $\ind(P) \in K_\ast(C_u^\ast(M))$ by first constructing $[P] \in K_\ast^u(M)$ and then mapping it by the rough assembly map to $K_\ast(C_u^\ast(M))$. But constructing $\ind(P)$ directly by the same procedure as above for Dirac operators, we get a problem: we only know from Lemma \ref{lem:exp(itP)_quasilocal} that $e^{itP}$ is a quasilocal operator with linearly decaying dominating function. Since we currently don't have an answer for the above Question~\ref{quesnksdio2}, we can not guarantee that this direct construction would produce a rough index class of $P$ which lives in the $K$-theory of the uniform Roe algebra, i.e., which is approximable by finite propagation operators.

In \cite{engel_rough} the author introduced a smooth subalgebra of the uniform Roe algebra consisting of those operators whose dominating functions is super-polynomially fast decaying. Since we showed in Lemma \ref{lem:exp(itP)_quasilocal} that $e^{itP}$ has a linearly decaying dominating function, the question is whether we can improve this result and so make it amenable to the techniques of \cite{engel_rough}. Note that Lemma \ref{lem:exp(itP)_quasilocal} does not assume anything on the dominating function of $P$, i.e., one might hope that one can get better rates of decay for the dominating function of $e^{itP}$ if one assume that $P$ itself already has good decay of its dominating funtion.

\begin{question}\label{question:wave_dominating_function}
Let $P$ be a symmetric and elliptic uniform pseudodifferential operator.

Does the dominating function of $e^{itP}$ have super-polynomial decay? Maybe if we assume that $P$ has finite propagation or a super-polynomially decaying dominating function?
\end{question}

\subsection*{Further questions about uniform pseudodifferential operators}

We know that the principal symbol map $\sigma^k$ induces an isomorphism of vector spaces $\UPsiDO^{k-[1]}(E,F) \cong \Symb^{k-[1]}(E,F)$ for all $k \in \IZ$ and vector bundles $E$, $F$ of bounded geometry. For the case $k = 0$ and $E = F$ we furthermore know from Proposition \ref{prop:PsiDOs_filtered_algebra} that $\UPsiDO^{0-[1]}(E)$ is an algebra, and $\sigma^0$ will be an isomorphism of algebras.

In the case that the manifold $M$ is compact, it is known that $\sigma^0$ is continuous against the quotient norm\footnote{Which is induced from the operator norm on $\Psi \mathrm{DO}^0(E) \subset \IB(L^2(E))$. Since for $M$ compact we have $\overline{\Psi \mathrm{DO}^{-1}(E)} = \IK(L^2(E))$, the quotient norm on $\Psi \mathrm{DO}^{0-[1]}(E)$ is called the \emph{essential norm}.} on $\Psi \mathrm{DO}^{0-[1]}(E)$ and therefore $\sigma^0$ will induce an isomorphism of $C^\ast$-algebras $\overline{\Psi \mathrm{DO}^{0-[1]}(E)} \cong \overline{\Symb^{0-[1]}(E)}$.

\begin{question}
Let $M$ be a non-compact manifold of bounded geometry. Does $\sigma^0$ induce an isomorphism of $C^\ast$-algebras $\overline{\UPsiDO^{0-[1]}(E)} \cong \overline{\Symb^{0-[1]}(E)}$?
\end{question}

To show this we would have to compare the quotient norms on $\UPsiDO^{0-[1]}(E)$ and on $\Symb^{0-[1]}(E)$. The first to prove similar results in the compact case were Seeley in \cite[Lemma 11.1]{seeley} and Kohn and Nirenberg in \cite[Theorem A.4]{kohn_nirenberg}, and two years later H{\"o}rmander provided in \cite[Theorem 3.3]{hormander_ess_norm} a proof of this for his class $S_{\rho, \delta}^0$ with $\delta < \rho$ of pseudodifferential operators of order $0$. Maybe one of these proofs generalizes to our case of uniform pseudodifferential operators on open manifolds.

The main technical part in the proof of Theorem \ref{thm:elliptic_symmetric_PDO_defines_uniform_Fredholm_module} that a uniform pseudodifferential operator defines a class in uniform $K$-homology was to show that the operator $\chi(P)$ is uniformly pseudolocal for $\chi$ a normalizing function. In Proposition \ref{prop:PDO_order_0_l-uniformly-pseudolocal} we have shown that uniform pseudodifferential operators of order $0$ are automatically uniformly pseudolocal. So if we could show that the operator $\chi(P)$ is a uniform pseudodifferential operator of order $0$, the proof of Theorem \ref{thm:elliptic_symmetric_PDO_defines_uniform_Fredholm_module} would follow immediately.

\begin{question}
Under which conditions on the function $f$ (or the operator $P$) will be $f(P)$ again a uniform pseudodifferential operator?
\end{question}

For a compact manifold $M$ there are quite a few proofs that under certain conditions functions of pseudodifferential operators are again pseudodifferential operators: the first one to show such a result was Seeley \cite{seeley_complex_powers}, where he proved it for complex powers of elliptic classical pseudodifferential operators. It was then extended by Strichartz \cite{strichartz} from complex powers to symbols in the sense of Definition \ref{defn:symbols_on_R}, and from classical operators to all of H{\"o}rmander's class $S^k_{1,0}(M)$. And last, we mention the result \cite[Theorem~8.7]{dimassi_sjostrand} of Dimassi and Sj{\"o}strand for $h$-pseudodifferential operators in the semi-classical setting.

Now if we want to establish similar results in our setting, we get quite fast into trouble: e.g., the proof of Strichartz does not generalize to non-compact manifolds. He crucially uses that on compact manifolds we may diagonalize elliptic operators, which is not at all the case on non-compact manifolds (consider, e.g., the Laplace operator on Euclidean space). Looking for a proof that may be generalized to the non-compact setting, we stumble over Taylor's result from \cite[Chapter XII]{taylor_pseudodifferential_operators}. There he proves a result similar to Strichartz' but with quite a different proof, which may be possibly generalized to non-compact manifolds. An evidence for this is given by Cheeger, Gromov and Taylor in \cite[Theorem 3.3]{cheeger_gromov_taylor}, since this is exactly the result that we want to prove for our uniform pseudodifferential operators, but in the special case of the operator $\sqrt{- \Delta}$, and their proof is a generalization of the one from the above cited book of Taylor. So it seems quite reasonable that we may probably extend the result of Cheeger, Gromov and Taylor to all uniform pseudodifferential operators in our sense.

The above ideas were already used by Kordyukov \cite{kordyukov_2} to derive $L^p$-estimates for functions of certain elliptic uniform pseudodifferential operators. Furthermore, in the same article he also used ideas surrounding the geometric optics equation, which are Taylor's main tool in \cite[Chapter VIII]{taylor_pseudodifferential_operators}, to show that functions of elliptic pseudodifferential operators with positive scalar principal symbol are again pseudodifferential operators.

Beals and Ueberberg both gave in their articles \cite{beals} and \cite{ueberberg} characterizations of pseudodifferential operators via certain mapping properties of these operators from the Schwartz space to its dual. From that they derived that the inverse, if it exists, of a pseudodifferential operator of order $0$ is again a pseudodifferential operator.

\begin{question}
Does there exists a similar characterization of uniform pseudodifferential operators on manifolds of bounded geometry as the one in \cite{beals} and \cite{ueberberg} by Beals and Ueberberg?
\end{question}

\bibliography{./Bibliography_Classes_of_elliptic_uniform_PDOs}
\bibliographystyle{amsalpha}

\end{document}